\documentclass[oneside,a4paper,11pt,reqno]{amsart}

\usepackage[T1]{fontenc}

\usepackage[a4paper, scale={0.72,0.74}, marginratio={1:1}, footskip=7mm, headsep=10mm]{geometry}

\usepackage{verbatim}
\usepackage{amsmath} \allowdisplaybreaks
\usepackage{amsfonts}
\usepackage{color}
\usepackage{makecell}
\usepackage{bm}

\allowdisplaybreaks
\usepackage{tikz}
\usetikzlibrary{topaths,fit,backgrounds,shapes.geometric}
\usetikzlibrary{graphs}
\usetikzlibrary{calc}
\usepackage{multirow}
\usepackage{enumerate}
\usepackage{cases}
\usepackage{graphicx}
\usepackage{xcolor}
\usepackage{geometry}
\usepackage{indentfirst}
\usepackage{amsfonts}
\usepackage{mathrsfs}
\usepackage{multicol}
\usepackage{url}
\usepackage{float}
\usepackage{color}
\usepackage{verbatim}
\usepackage{stackengine}
\stackMath
\usepackage[dvipsnames]{xcolor}
\usepackage[normalem]{ulem}
\usepackage{cancel}

\usepackage{fancyhdr}
\usepackage{caption}
\usepackage{mathtools}
\usepackage{dsfont}
\usepackage{graphicx}
\usepackage{esdiff}
\usepackage{physics}
\usepackage{subcaption}
\DeclarePairedDelimiter\nm{\lVert}{\rVert}
\usepackage{nicematrix}

\definecolor{citeorange}{RGB}{200,120,0}
\definecolor{citegreen}{RGB}{0,128,128}

\usepackage{hyperref}
\hypersetup{
  colorlinks=true,
  linkcolor=citegreen,
  citecolor=citeorange,
  urlcolor=
}

\usepackage{amssymb}

\numberwithin{equation}{section}

\newtheorem{theorem}{Theorem}[section]
\newtheorem{proposition}[theorem]{Proposition}
\newtheorem{corollary}[theorem]{Corollary}
\newtheorem{lemma}[theorem]{Lemma}
\theoremstyle{definition}
\newtheorem{Definition}[theorem]{Definition}
\newtheorem{Assumption}[theorem]{Assumption}
\newtheorem{Remark}[theorem]{Remark}
\newtheorem{Example}[theorem]{Example}

\newenvironment{myenumerate}{%
\renewcommand{\theenumi}{(\roman{enumi})}%
\renewcommand{\labelenumi}{\theenumi}%
\begin{list}{\labelenumi}
	{%
	\setlength{\itemsep}{0.4em}%
	\setlength{\topsep}{0.5em}%
	\setlength\leftmargin{2.45em}%
	\setlength\labelwidth{2.05em}%
	\setlength{\labelsep}{0.4em}%
	\usecounter{enumi}%
	}%
	}%
{\end{list}
}
\renewenvironment{enumerate}{
\begin{myenumerate}}%
{\end{myenumerate}}

\newenvironment{myitemize}{%
\begin{list}{$\bullet$}%
 	{%
	\setlength{\itemsep}{0.4em}%
	\setlength{\topsep}{0.5em}%
	\setlength\leftmargin{2.45em}%
	\setlength\labelwidth{2.05em}%
	\setlength{\labelsep}{0.4em}%
	}%
	}%
{\end{list}}
\renewenvironment{itemize}{
\begin{myitemize}}%
{\end{myitemize}}

\def \E {\mathbb{E}}
\def \N  {\mathbb{N}} 

\def \R  {\mathbb{R}}

\renewcommand{\epsilon}{\varepsilon}

\newcommand{\de}{\mathrm{d}}

\newcounter{para}

\usepackage[style=alphabetic, backend=bibtex8, citestyle=alphabetic, firstinits=true, maxnames=99, maxalphanames=3, sorting=nyt]{biblatex}
\DeclareFieldFormat[article, inbook, incollection, inproceedings, patent, thesis, unpublished]{title}{#1} 
\renewbibmacro{in:}{}
\appto{\bibsetup}{\sloppy}

\begin{document}

\title[Non-uniform random hypergraphs]{
Semicircle laws with combined variance \\ for non-uniform Erd\H{o}s--R\'enyi hypergraphs}

\author[L.~Avena]{Luca Avena} 
\address{Dipartimento di Matematica e Informatica ``Ulisse Dini'', Universit\`a degli Studi di Firenze, Viale Morgagni 67/A, 50134, Florence, Italy}
\email{luca.avena@unifi.it}

\author[E.~Bisi]{Elia Bisi}
\address{Dipartimento di Matematica e Informatica ``Ulisse Dini'', Universit\`a degli Studi di Firenze, Viale Morgagni 67/A, 50134, Florence, Italy}
\email{elia.bisi@unifi.it}

\author[E.~Bordiga]{Eleonora Bordiga}
\address{Dipartimento di Matematica e Informatica ``Ulisse Dini'', Universit\`a degli Studi di Firenze, Viale Morgagni 67/A, 50134, Florence, Italy}
\email{eleonora.bordiga@gmail.com}

\begin{abstract}
We consider Erd\H{o}s--R\'enyi-type random hypergraphs that are non-uniform, in the sense that hyperedges of different sizes may coexist, and inhomogeneous, in that connection probabilities may depend on the hyperedge size.
All parameters are allowed to scale with the hypergraph size.
We study the random adjacency matrix whose $(u,v)$-entry counts the number of hyperedges containing both vertices $u$ and $v$, and characterize its expected limiting spectral distribution in terms of the connection probabilities and the hyperedge sizes.
We provide a Pastur-type condition, in the sense of Chatterjee (2005), under which the matrix can be Gaussianized, as well as a more restrictive but simpler sufficient condition in terms of the generalized average degree of the model.
As a second main result, based on such a Gaussianization, we characterize the limiting spectral distributions under non-sparse conditions as semicircle laws with an explicit parametric variance.
The latter can be expressed as a convex combination of the variances arising in the uniform cases, with coefficients determined by the trade-off between the different sources of inhomogeneity.
\end{abstract}
\maketitle
\tableofcontents

\section{Introduction}

\subsection{Background}

Complex networks provide a mathematical framework for modeling systems {with} many interacting components.
Traditional models are typically based on graphs; however, many real-world systems exhibit interactions that extend beyond pairwise relations.
Such higher-order interactions naturally appear, for instance, in scientific collaboration~\cite{collaboration1}, chemical reactions~\cite{chemistry1, chemistry2, chemistry3}, and ecology~\cite{ecology1}.
This has motivated a growing interest in hypergraph-based models within network science. 

A hypergraph is a set of vertices endowed with a collection of subsets of vertices, called hyperedges.
Vertices represent system constituents, while hyperedges capture multi-body interactions among them.
A hypergraph is said to be uniform if all its hyperedges have the same size, and non-uniform otherwise.
Thus, graphs are uniform hypergraphs with (hyper)edges of size $2$.

Spectral theory provides powerful tools for studying the structure of networks.
In the graph setting, the spectra of connectivity matrices, such as adjacency and Laplacian matrices, encode important geometric information about the underlying structure.
Extending spectral methods from graphs to hypergraphs is currently an active area of research.
Unlike in the graph case, there is no canonical choice of linear algebraic objects associated with hypergraphs.
To uniquely encode the structure of a hypergraph, 
incidence matrices and tensors are commonly considered.
See, e.g.,~\cite{incidence_matrix} for the incidence matrix of general hypergraphs, and~\cite{Cooper1,Cooper2} for adjacency tensors in the uniform case.
A non-uniform version of the adjacency tensor was proposed in~\cite{Banerjee}.
Several other extensions of adjacency and Laplacian matrices have been introduced; see, e.g.,~\cite{Pearson_Zhang,laplacian_tensor}.
However, due to the computational complexity of spectral problems for tensors, which are largely NP-hard~\cite{np_hard_tensors}, matrix-based representations have also been explored, for both adjacency and Laplacian operators~\cite{Laplacian_matrix, Battiston_et_al, Banerjee2, ipergrafiuniformi}.
For example, a common approach is to consider the matrix whose $(u,v)$-entry counts the number of hyperedges containing both vertices $u$ and $v$.
This corresponds to projecting the hypergraph onto a multigraph, where an edge is placed between two vertices for each hyperedge they share.
This matrix is sometimes referred to as the \emph{weighted clique expansion} or \emph{co-occurrence projection matrix}~\cite{hypergraph_projection}.

Random models play a central role in the study of large complex networks.
The classical Erd\H{o}s--R\'enyi random graph has been extensively investigated, ranging from connectivity to spectral properties, and exhibits several fundamental phase transition phenomena.
For results on connectivity phase transitions, we refer the reader to~\cite[Chapter 4]{Van_der_Hofstad}.
For limiting spectral distribution in different regimes, see, e.g.,~\cite{erclassico} for the non-sparse case and~\cite{bauer_and_golinelli, khorunzhy2004eigenvalue, Bordenave_Lelarge_Salez, Salez, Bordenave_Sen_Virag, Coste_and_Salez} for the sparse case.
This problem has also been studied for inhomogeneous Erd\H{o}s--R\'enyi random graphs; see \cite{inhomogeneous_sparse} for the sparse case and \cite{inhomogeneous_non_dense} for the {intermediate (neither sparse nor dense) regime}.
Extensions of random graph models to hypergraphs have been actively studied in recent years; see, e.g.,~\cite{randomhypergraphs3, randomhypergraphs1, ipergrafiuniformi, randomhypergraphs2}.

Most of the existing literature on hypergraphs has focused on the uniform case.
However, this assumption may be restrictive when modeling heterogeneous systems, where interactions of different sizes may coexist.
For this reason, non-uniform hypergraphs provide more realistic and flexible models, as they can better account for inhomogeneities in the network.
In this paper, we study a non-uniform Erd\H{o}s--R\'enyi-type random model, parametrized by several quantities, namely the number of different hyperedge classes, the hyperedge sizes, and the connection probabilities as functions of the size.
This level of generality allows our framework to capture a wide range of random hypergraph structures through suitable choices of the parameters.
For this model, we analyze the asymptotic spectral distribution of the associated random adjacency matrices.

\subsection{Content and organization of the paper}
In Section~\ref{sec_model}, we provide a general framework for the study of non-uniform hypergraphs and their adjacency matrices.
Moreover, we define an Erd\H{o}s--R\'enyi-type random model, in which hyperedges are included independently with probabilities depending on their size.
Both the hyperedge sizes and the connection probabilities may scale with the dimension of the hypergraph.

In Section~\ref{sec_results}, we establish sufficient conditions under which the entries of the adjacency matrices can be replaced by Gaussian random variables (Theorem~\ref{thm_Gaussianization}) without affecting the limiting spectral distribution.
We provide both a Pastur-type condition, in the sense of Chatterjee~\cite{Chatterjee}, and a more restrictive but simpler sufficient condition in terms of the generalized average degree of the model.
As a consequence of this Gaussianization result, the expected empirical spectral distribution converges to a semicircle law (Theorem~\ref{thm_semicircle}).
The variance of the limiting distribution is a convex combination of the variances arising in the uniform cases, with coefficients determined by the joint scaling of hyperedge sizes and connection probabilities.

In Section~\ref{sec_discussion}, we examine in greater detail the case where hyperedges can take only two distinct sizes.
We identify \emph{dominant regimes}, i.e., regimes in which a single class of hyperedges determines the limiting spectral measure, and \emph{balanced regimes}, in which both classes contribute.
In addition, we discuss several consequences and examples of Theorem~\ref{thm_semicircle}.
Corollary~\ref{corollary_linear_sizes} addresses the case in which hyperedge sizes grow linearly with the dimension of the hypergraph.
Corollary~\ref{corollary_fixed_sizes} considers models with fixed hyperedge sizes; in particular, superpositions of classical Erd\H{o}s--R\'enyi random graphs are analyzed in Corollary~\ref{superposition_result}.
In Corollary~\ref{corollary_uniform_hypergraphs}, we recover the results of Ding and Jiang~\cite{erclassico} and Mukherjee et al.~\cite{ipergrafiuniformi} for the special cases of graphs and uniform hypergraphs, respectively.

Finally, Section~\ref{proofs} contains the proofs of the main results.
Inspired by~\cite{ipergrafiuniformi}, we establish the Gaussianization result (Theorem~\ref{thm_Gaussianization}) by adapting Chatterjee's extension of the Lindeberg replacement argument to our setting.
This yields a bound on the distance between the Stieltjes transform of the {expected empirical spectral distribution of the adjacency matrix of the model under consideration} and that of its Gaussianized counterpart (Lemma~\ref{applicazioneChatterjee}).
Using standard probabilistic inequalities, we then derive sufficient conditions ensuring that this error vanishes, namely the Pastur-type and non-sparsity conditions (Propositions~\ref{PasturGaussianization} and~\ref{prop_D1_D2}).
The proof of Theorem~\ref{thm_semicircle}, on the other hand, relies on the representation {used} in~\cite{ipergrafiuniformi}, which expresses the Gaussianized adjacency matrix as a finite-rank perturbation of a GOE matrix.
Our main contribution is to analyze how the different sources of non-uniformity and inhomogeneity in the model influence key asymptotic quantities, such as the conditions ensuring convergence of the {expected empirical spectral distribution} to the semicircle law and the variance of the limiting measure.

The study of the bulk spectrum constitutes a first step toward understanding the spectral properties of non-uniform random hypergraphs.
In this work, we characterize the expected limiting spectral distribution in the non-sparse regime.
Identifying the limiting measure in the sparse case requires different tools, most notably the theory of local weak convergence introduced by Benjamini and Schramm~\cite{Benjamini_Schramm}.
In particular, Bordenave and Lelarge~\cite{Bordenave_Lelarge} established a connection between this framework and the analysis of the limiting spectral distribution of graphs.
For sparse uniform hypergraphs, the identification of the limiting spectral measure via local weak convergence was carried out by Adikhari and Parui \cite{sparseuniform}.
The analysis of the local weak limit of the present non-uniform model is left for future work.
See Remark~\ref{rem:SparseCase} for a more detailed discussion and a review of the relevant literature on this point.

\subsection{Asymptotic notation}

Throughout, we adopt the following asymptotic notation.
For functions $f,g\colon \mathbb{N} \rightarrow \mathbb{R}$, we write: 
\begin{itemize}
   \item $f(n)=O(g(n))$ if there exist $c>0$ and $n_0 \in \mathbb{N}$ such that $|f(n)|\leq c |g(n)|$, $ \forall n \geq n_0$.
   \item $f(n)=\Theta(g(n))$ if $f(n)=O(g(n))$ and $g(n)=O(f(n))$.
   \item $f(n) \ll g(n)$, or $g(n) \gg f(n)$, or $f(n) = o(g(n))$, if $\lim_{n \to \infty} \frac{f(n)}{g(n)}=0.$
   \item $f(n) \sim g(n)$, if 
      $\lim_{n \to \infty} \frac{f(n)}{g(n)}=1$.
   \end{itemize}

\section{Model}
\label{sec_model}

\subsection{Non-uniform hypergraphs and adjacency matrices}

A \emph{hypergraph} of dimension $n$ is a pair $\mathcal{H}=(V,E)$, where $V=\{1,\dots,n\}$ is the \emph{vertex set} and $E \subseteq 2^{V}$ is the \emph{hyperedge set}.
We assume that each hyperedge is a subset of vertices of cardinality at least $2$ (i.e., no self-loops are allowed).
For brevity, we refer to a hyperedge of size $r$ as an \emph{$r$-edge}.
A hypergraph is called $r$-\emph{uniform} if all of its hyperedges have size $r$, for some $r\geq 2$, and \emph{non-uniform} otherwise.
In this sense, classical undirected graphs correspond to $2$-uniform hypergraphs.

Throughout, we assume that $\mathcal{H}$ is a non-uniform hypergraph whose hyperedges can take sizes $r_1,\dots,r_k$, for some $k \in \mathbb{N}$.
One can decompose $\mathcal{H}$ into its uniform components: if $E^{(i)}$ denotes the set of its $r_i$-edges, then $\mathcal{H}^{(i)}=(V,E^{(i)})$ is an $r_i$-uniform subhypergraph on the same vertex set, and $E^{(1)}, \dots, E^{(k)}$ form a partition of $E$.

It is well known that the structure of a graph can be encoded in a matrix, called the \emph{adjacency matrix} of the graph.
However, for hypergraphs, the appropriate objects are higher-order tensors.
Since tensors are often difficult to handle, one typically studies suitable matrix representations obtained via tensor contractions instead.
Mukherjee et al.~\cite{ipergrafiuniformi} studied an adjacency matrix associated with a uniform hypergraph, obtained as a contraction of a tensor and defined as
\begin{equation*}
A_{uv} =
\begin{cases}
   \sum_{e \in E} \mathds{1}(u,v \in e) & \text{if } u \neq v,\\
   0 & \text{otherwise}.
\end{cases}
\end{equation*}
In other words, for $u \neq v$, the entry $A_{uv}$ counts the number of hyperedges in $\mathcal{H}$ containing both vertices $u$ and $v$, while the diagonal entries are zero, as self-loops are not allowed.
A natural way to extend this definition to non-uniform hypergraphs is to define the adjacency matrix of $\mathcal{H}$ as the sum of the adjacency matrices $A^{(1)}, \dots, A^{(k)}$ of its $k$ uniform subhypergraphs, as in~\cite{uniform_decomp_1}:
\begin{equation} \label{adjacency_matrix}
   A \coloneqq  \sum_{i=1}^{k} A^{(i)}.
\end{equation}
For $u \neq v$, the $(u,v)$-entry of this matrix is given by
\begin{align*}
   A_{uv}  = \sum_{e \in E} \mathds{1}(u,v \in e) 
   = \sum_{i=1}^{k} \sum_{e \in E^{(i)}} \mathds{1}(u,v \in e).
\end{align*}
Note that, since $|V| = n$, the matrix $A$ is an $n\times n$ symmetric matrix with non-negative integer entries.
This contrasts with adjacency matrices of simple graphs, whose entries take values in $\{0,1\}$.
Nevertheless, $A$ admits a convenient representation as a linear combination of matrices of this latter type, as we show next.

Consider the \emph{$(r_1,\dots,r_k)$-complete hypergraph}, i.e., the hypergraph that contains all possible hyperedges of sizes $r_1,\dots,r_k$.
The total number of its hyperedges is 
\begin{equation*}
   M=\sum_{i=1}^{k} \binom{n}{r_i}.
\end{equation*}
Let $\{e_1,\dots,e_M\}$ be an enumeration of these hyperedges, ordered so that hyperedges of the same class occupy consecutive positions.
More precisely, define
\begin{equation*}
M_i \coloneqq \sum_{j=1}^{i} \binom{n}{r_j}, \qquad M_0 \coloneqq 0,
\end{equation*}
and suppose that $e_{M_{i-1}+1},\dots,e_{M_i}$ are the $r_i$-edges.

Now consider again an arbitrary hypergraph $\mathcal{H}$, decomposed into its $k$ uniform components $\mathcal{H}^{(1)},\dots,\mathcal{H}^{(k)}$.
Then, the adjacency matrix of $\mathcal{H}^{(i)}$ can be written, as in~\cite{ipergrafiuniformi}, as
\begin{equation} \label{linear_comb}
    A^{(i)} = \sum_{l=M_{i-1}+1}^{M_i}h_lQ_l, \qquad 1\leq i\leq k,
 \end{equation}
where 
\begin{itemize}
   \item $h_l \coloneqq \mathds{1}(e_l \in E)$ indicates whether the hyperedge $e_l$ is present in $\mathcal{H}$ (if $\mathcal{H}$ is an $(r_1,\dots,r_k)$-complete hypergraph, then $h_l \equiv 1$ for all $l$);
   \item $Q_l$ is the $\{0,1\}$-matrix of size $n$, with entries $(Q_l)_{uv} \coloneqq \mathds{1}(u,v \in e_l)$, encoding which vertices belong to $e_l$.
\end{itemize}
By~\eqref{linear_comb}, the full adjacency matrix~\eqref{adjacency_matrix} can thus be written as
\begin{equation} \label{adjacency_matrix_representation}
    A = \sum_{l=1}^{M}h_lQ_l.
\end{equation}

We emphasize that the above considerations apply in both deterministic and random settings. However, as we will shortly see, the representation~\eqref{adjacency_matrix_representation} is particularly useful for Erd\H{o}s--R\'enyi-type hypergraphs, where the matrices $Q_l$ are deterministic, while all randomness is captured by the coefficients $h_l$.

\subsection{Non-uniform inhomogeneous Erd\H{o}s--R\'enyi hypergraphs}

We consider an Erd\H{o}s--R\'enyi-type model, in which each hyperedge of the $(r_1,\dots,r_k)$-complete hypergraph is included independently, with probability depending on its size:
\begin{equation} \label{Bernoulli}
    h_{l} \sim \mathrm{Ber}(p_i) \qquad \text{if }  |e_l| = r_i, \qquad 1\leq l\leq M.
\end{equation}
Namely, for $1 \leq i \leq k$, each $r_i$-edge is present with probability $p_i$; since, in general, $p_1,\dots,p_k$ are distinct, this introduces a form of inhomogeneity in the random model, in addition to its non-uniformity.
Additionally, we allow both the connection probabilities $p_i=p_i(n)$ and the hyperedges sizes $r_i=r_i(n)$ to depend on $n$.
However, for the sake of simplicity, we assume that $k$ is fixed and does not depend on $n$.

\begin{Definition}[$({n,}\bm{r},\bm{p})$-Erd\H{o}s--R\'enyi hypergraph] \label{our_model}
{Fix $k\in \N$.
Let $\bm{r}=(r_i)_{i=1}^{k}\in \N^k$ and $\bm{p}=(p_i)_{i=1}^{k}\in [0,1]^k$, both possibly depending on $n$.}
We refer to the stochastic model introduced above as the \emph{$(n,\bm{r},\bm{p})$-Erd\H{o}s--R\'enyi hypergraph} and we denote it by $\mathcal{H}(n,\bm{r},\bm{p})$.
\end{Definition}

The adjacency matrix $A$ of $\mathcal{H}(n,\bm{r},\bm{p})$ is a symmetric random matrix, whose off-diagonal entries are identically distributed.
Indeed, for all $i \in \{1,\dots,k\}$, any two distinct vertices $u$ and $v$ can belong to $\binom{n-2}{r_i-2}$ possible $r_i$-edges, each of which is included independently in the random hypergraph with probability $p_i$, so that
\begin{equation*}
    A_{uv}^{(i)} \sim \mathrm{Bin}\!\left(\binom{n-2}{r_i-2},\,p_i \right), \qquad u\neq v.
\end{equation*}
Therefore, by (\ref{adjacency_matrix}), each off-diagonal entry of the full adjacency matrix $A$ is a sum of $k$ independent binomial random variables {with the parameters specified above}.
Note that, if the $p_i$ are distinct, $A_{uv}$ does not itself follow a binomial distribution.

The mean and variance of each off-diagonal entry are given by
\begin{equation}
    \mu \coloneqq
    \mathbb{E}\!\left[A_{uv} \right]
    = \sum_{i=1}^{k} \binom{n-2}{r_i-2}p_i, 
    \qquad
    \sigma^2 \coloneqq 
    \mathrm{Var}[A_{uv}]
    = \sum_{i=1}^{k} \binom{n-2}{r_i-2}\sigma_{i}^2,
    \qquad  u \neq v, \label{sigmasquare}
\end{equation}
where 
\begin{equation} \label{sigma_i^2}
   \sigma_i^2 \coloneqq p_i(1-p_i), \qquad 1 \leq i \leq k,
\end{equation}
are the variances of the Bernoulli random variables in~\eqref{Bernoulli}.

A remarkable feature of the adjacency matrix $A$ is that its entries are not independent (unless $r_i=2$ for all $1 \leq i \leq k$), in contrast to usual Wigner matrices:
\begin{equation} \label{covariance}
    \mathrm{Cov}(A_{uv},A_{u'v'})=
\begin{cases}
   \displaystyle \sum_{i=1}^{k} \binom{n-4}{r_i-4}\sigma_i^2 & \text{if } |\{u,v,u',v'\}|=4,\\[15pt]
   \displaystyle \sum_{i=1}^{k} \binom{n-3}{r_i-3}\sigma_i^2 & \text{if } |\{u,v,u',v'\}|=3.
\end{cases}
\end{equation}
Such random matrices are sometimes referred to as \emph{dependent} (or \emph{correlated}) \emph{Wigner matrices}.

We now describe some illustrative submodels that are encompassed by our framework.

\begin{Example}[Homogeneous case]
In the homogeneous one, all connection probabilities are taken to be equal, i.e., $\bm{p}=(p,\dots,p)$.
In this case, $h_l \overset{\text{i.i.d.}}{\sim} \mathrm{Ber}(p)$ and $A_{uv} \sim \mathrm{Bin}\!\left(\sum_{i=1}^{k}\binom{n-2}{r_i-2},\,p \right)$, $u \neq v$.
\end{Example}

\begin{Example}[Graphs and uniform hypergraphs]
When $k=1$, the model allows a single {hyperedge} size $r$ and a single connection probability $p$.
In this case, we recover the $r$-uniform random hypergraph model $\mathcal{H}(n,r,p)$.
In particular, $\mathcal{H}(n,2,p)$ coincides with the classical Erd\H{o}s--R\'enyi random graph $\mathcal{G}(n,p)$.
\end{Example}

\begin{Example}[Superposition of Erd\H{o}s--R\'enyi random graphs] \label{def_superposition}
We now consider the degenerate case in which all (hyper)edge sizes {are equal to $2$}, i.e., $\bm{r}=(2,\dots,2)$.
This choice yields a \emph{superposition of $k$ Erd\H{o}s--R\'enyi random graphs}.
This corresponds to a random multigraph in which each pair of vertices may be connected by at most $k$ edges (see Figure~\ref{fig:superposition}).
The off-diagonal entries of the adjacency matrix are given by a sum of $k$ independent Bernoulli random variables:
    \begin{equation*}
        A_{uv} = \sum_{i=1}^{k} A_{uv}^{(i)} \in \{0,1,\dots,k\}, \qquad  u \neq v,
    \end{equation*}
where $A_{uv}^{(i)} \sim \mathrm{Ber}(p_i)$.
In the homogeneous case, when $p_i=p$ for all $i$, we have $A_{uv} \sim \mathrm{Bin}(k,p)$.
\end{Example}

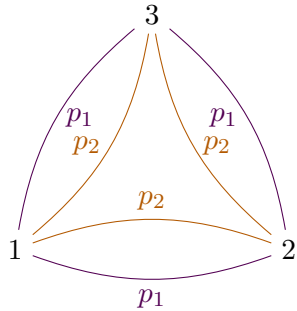
\begin{figure}
\begin{center}
\begin{tikzpicture}[scale=1.8]

  \node (1) at (0,0) {$1$};
  \node (2) at (2,0) {$2$};
  \node (3) at (1,1.732) {$3$};

  \draw[orange!70!black] (1) to[bend left=20] node[midway, above] {$p_2$} (2);
  \draw[violet!65!black] (1) to[bend right=20] node[midway, below] {$p_1$} (2);

  \draw[orange!70!black] (2) to[bend left=20] node[midway, right] {$p_2$} (3);
  \draw[violet!65!black] (2) to[bend right=20] node[midway, left] {$p_1$} (3);

  \draw[orange!70!black] (3) to[bend left=20] node[midway, left] {$p_2$} (1);
  \draw[violet!65!black] (3) to[bend right=20] node[midway, right] {$p_1$} (1);

\end{tikzpicture}
\end{center}
\caption{Superposition of two classical Erd\H{o}s--R\'enyi random graphs ($k=2$) on 3 vertices.}
\label{fig:superposition}
\end{figure}

\subsection{Normalization of the adjacency matrix}
It is standard to analyze the spectrum of a normalized version of the adjacency matrix.
We follow this approach by centering $A$ (i.e., subtracting its mean matrix $\mathbb{E}\!\left[A \right]$) and scaling by the factor $\sqrt{n\sigma^2}$.
In other words, we consider
\begin{equation} \label{def_H}
     H_n \coloneqq \frac{A-\mathbb{E}\!\left[A \right]}{\sqrt{n\sigma^2}}.
\end{equation}
By construction, we have $\mathbb{E}\!\left[(H_n)_{uv} \right]=0$ and $\mathrm{Var}[(H_n)_{uv}]=1/n$, for $u \neq v$, while $(H_n)_{uu} = 0$ for $1 \leq u \leq n$.
Many properties of the original adjacency matrix $A$ are inherited by $H_n$, such as its symmetry and the fact that the entries are identically distributed but not independent.

\begin{Remark} \label{EESD_second_moments}
The normalization~\eqref{def_H} ensures that the second moment of $H_n$ remains bounded and non-vanishing.
Indeed, the second moment of any random real symmetric matrix $B$ is given by
\begin{equation*}
   m_2(B)  
   = \frac{1}{n} \mathbb{E}\!\left[\mathrm{Tr}B^2\right]
   = \frac{1}{n} \sum_{1 \leq u, v \leq n}\mathbb{E}\!\left[B_{uv}^2\right].
\end{equation*}
Applying this equality to $B=A-\mathbb{E}\!\left[A \right]$, we obtain
\begin{align*}
   m_2(A-\mathbb{E}\!\left[A \right]) 
   = \frac{1}{n} \sum_{1 \leq u,v \leq n} \mathrm{Var}[A_{uv}].
\end{align*}
The diagonal entries of $A$ are zero, while its off-diagonal entries all have the same variance $\sigma^2$ (see~\eqref{sigmasquare}), so that
\begin{equation*}
   m_2(H_n) 
   = \frac{n-1}{n}.
\end{equation*}
\end{Remark}

\begin{Remark}[Complete and empty hypergraphs]
For trivial parameter choices, namely when \emph{all} the parameters $p_1,\dots,p_k$ are either $0$ or $1$, the scaled matrix $H_n$ is not defined, as $\sigma^2=0$.
In these cases $A$ is deterministic and it equals a non-negative multiple of the adjacency matrix of the complete graph on $n$ vertices.
Henceforth, these cases will be excluded.
\end{Remark}

Notice that, by the representation of $A$ in~\eqref{adjacency_matrix_representation}, we have
\begin{equation*}
   A-\mathbb{E}\!\left[A \right]
   = \sum_{l=1}^M X_lQ_l,
 \end{equation*}
where $\bm{X}:=(X_1,\dots,X_M)$ is the vector with components
\begin{equation} \label{X}
   X_l \coloneqq h_l-\mathbb{E}\!\left[h_l \right].
\end{equation}
Note that these random variables satisfy the following properties:
\begin{itemize}
    \item $X_1,\dots,X_M$ are independent;
    \item $\mathbb{E}\!\left[X_l \right]=0$, for any $1 \leq l \leq M$;
    \item if $|e_l| = r_i$, then $\mathrm{Var}[X_l]= \sigma_{i}^2$ (as defined in~\eqref{sigma_i^2}).
\end{itemize}
In the following, we often write $H_n(\bm{X})$ to emphasize that $H_n$ can be viewed as a function of the random vector $\bm{X}$.

\section{Main results}
\label{sec_results}

\subsection{Notation}
For a random real symmetric $n\times n$ matrix $B_n$ with (real) eigenvalues $\lambda_1, \dots, \lambda_n$, we denote by $\mu_{B_n}$ its \emph{empirical spectral distribution} (ESD), i.e., the random probability measure that assigns mass to the eigenvalues proportionally to their relative frequencies:
\begin{equation*}
   \mu_{B_n} \coloneqq \frac{1}{n} \sum_{i=1}^n \delta_{\lambda_i},
\end{equation*}
where $\delta_{\lambda}$ denotes the Dirac measure at $\lambda \in \mathbb{R}$. 

We denote by $\bar{\mu}_{B_n}$ the \emph{expected empirical spectral distribution} (EESD) of $B_n$, i.e., the deterministic probability measure defined via its action on bounded continuous test functions as
\begin{equation*}
   \int_{\mathbb{R}} f(x) \, d\bar{\mu}_{B_n}(x)
   \coloneqq \mathbb{E}\!\left[ \int_{\mathbb{R}} f(x) \, d\mu_{B_n}(x) \right].
\end{equation*}

If, as $n \to \infty$, $\bar\mu_{B_n}$ converges weakly, we refer to the limit as the \emph{expected limiting spectral distribution} (ELSD) of $B_n$.
In our setting, the ELSD will always be a \emph{semicircle law with variance $s^2$}, i.e., the probability measure $\nu_{s^2}$ on $\mathbb{R}$ with density
\begin{equation*}
   \frac{1}{2\pi s^2}\sqrt{4 s^2-x^2}\,
   \mathds{1}({|x|\le 2s}).
\end{equation*}

\subsection{Overview of the results}
Our main object of interest is the bulk of the spectrum of $H_n(\bm{X})$ in the large-$n$ regime, where $H_n$ is the centered and scaled adjacency matrix, defined in~\eqref{def_H}, of the Erd\H{o}s--R\'enyi hypergraph $\mathcal{H}(n,\bm{r},\bm{p})$.
More specifically, our goal is to identify sufficient conditions under which the EESD of this ensemble converges to a semicircle law, and to determine its variance in terms of the model parameters $(\bm{r},\bm{p})$.
We proceed in two steps:
\begin{enumerate}
    \item \textbf{Gaussianization}: we identify sufficient conditions under which the vector $\bm{X}$ can be replaced by a Gaussian vector $\bm{Z}$ with the same mean and covariance matrix, with a negligible perturbation of the EESD of $H_n(\bm{X})$.
    \item \textbf{Convergence to the semicircle law in the Gaussian case}: we study the ELSD of the ensemble $H_n$ after replacing the random variables of our model with Gaussian ones.
    The advantage is that the asymptotic analysis becomes more tractable in the Gaussian setting, as the random matrix ensemble can be reduced to a Gaussian Orthogonal Ensemble (GOE), allowing us to exploit known results.
    The limiting variance is determined by the covariance structure of $H_n$, and thus depends on the model parameters.
\end{enumerate}

\subsection{Convergence to the semicircle law}
We begin by introducing the assumptions on the model parameters for the convergence of the EESD of $H_n(\textbf{X})$ to the semicircle law. 
The following regularity condition will be assumed throughout.
\begin{Assumption}[Regularity condition] \label{regularity_condition}
We assume that the following limits exist:
\begin{equation} \label{weights}
    w_i \coloneqq \lim_{n \to \infty} \frac{\binom{n-2}{r_i-2}\sigma_i^2}{\sum_{j=1}^{k}\binom{n-2}{r_j-2}\sigma_j^2}, \qquad 1\leq i\leq k.
\end{equation}
\end{Assumption}

For each $1 \leq i \leq k$, we introduce the following quantity, which, in analogy with graphs and uniform hypergraphs~\cite{ipergrafiuniformi}, can be interpreted as the \emph{average $r_i$-degree}, i.e., the average number of $r_i$-edges containing a given vertex:
\begin{equation} \label{defaveragedegrees}
    d_i \coloneqq \binom{n-1}{r_i-1} \, p_i, \qquad 1 \leq i \leq k.
\end{equation}
Indeed, any given vertex can belong to $\binom{n-1}{r_i-1}$ possible hyperedges of size $r_i$, each of which is included independently with probability $p_i$.
We work under the following assumption on the average degrees.
\begin{Assumption}[Non-sparsity condition] \label{non_sparsity_condition}
    Let $r_{\max} \coloneqq \max\{r_i \colon 1 \leq i \leq k\}$ and
    \begin{equation} \label{xi}
        \xi \coloneqq \sum_{i=1}^{k} w_i \frac{1}{r_i^2},
    \end{equation}
    where $w_1,\dots,w_k$ are defined in \eqref{weights}.
    We assume that the following non-sparsity condition holds:
    \begin{equation}
        \label{D}
         \frac{a_n}{b_n c_n}
        \rightarrow \infty,
    \end{equation}
    where 
    \begin{equation*}
        a_n = \left(\sum_{i=1}^{k} r_i d_i\right)^2, 
        \qquad b_n = r_{\mathrm{max}}^{16}\xi^2, 
        \qquad c_n = \sum_{i=1}^{k} \frac{d_i}{r_i}.
    \end{equation*}
\end{Assumption}
Notice that, in general, $\xi$ depends on $n$, as the hyperedge sizes $r_i$ do.

\begin{Remark}[Hyperedges of the same magnitude and non-sparsity condition] \label{same_magnitude_condition}
Suppose that $r_1,\dots,r_k$ have the same order of magnitude, i.e., there exists $r=r(n)$ such that $r_i = \Theta(r)$ for $1 \leq i \leq k$.
In this regime, we have $a_n=\Theta(r d_{\mathrm{tot}})^2$, $b_n = \Theta(r^{12})$, and $c_n = \Theta(d_{\mathrm{tot}}/r)$, where $d_{\mathrm{tot}} \coloneqq \sum_{i=1}^{k} d_i$ denotes the \emph{average total degree} of any vertex in the hypergraph.
Therefore, condition~\eqref{D} is equivalent to
\begin{equation*}
\frac{d_{\mathrm{tot}}}{r^9} \to \infty.
\end{equation*}
We deduce that, in the non-uniform setting with hyperedge sizes of the same magnitude, Assumption~\ref{non_sparsity_condition} holds if and only if there exists $i \in \{1,\dots,k\}$ such that 
\begin{equation*}
\frac{d_i}{r_i^{9}} \to \infty.
\end{equation*}
Notably, {for $k=1$}, this coincides with the non-sparsity condition for the convergence to the semicircle law obtained in~\cite{ipergrafiuniformi} for uniform hypergraphs.
\end{Remark}

We are now ready to state our main result. 

\begin{theorem}[Convergence to the semicircle law] \label{thm_semicircle}
Consider the Erd\H{o}s--R\'enyi hypergraph $\mathcal{H}(n,\bm{r},\bm{p})$ as in Definition \ref{our_model}.
For each $1 \leq i \leq k$, assume that $r_i/n \rightarrow c_i \in [0,1)$ and $1-p_i=\Theta(1)$.
Suppose moreover that the non-sparsity condition~\eqref{D} holds.
Then
\begin{equation*}
        \bar{\mu}_{H_n(\bm{X})} \xrightarrow{n\to\infty} \nu_{s^2} \qquad\text{weakly,}
\end{equation*}
where $\nu_{s^2}$ is the semicircle law with variance 
\begin{equation} \label{variance_semicircle}
s^2 = \sum_{i=1}^{k} w_i(1-c_i)^2,
\end{equation}
and $w_1,\dots,w_k$ are defined in~\eqref{weights}.
\end{theorem}

\begin{Remark}[Loss of variance of the semicircle law]
It is worth noting that, whenever $c_i$ and $w_i$ are both positive, for some $1 \leq i \leq k$, the variance of the limiting semicircle law is strictly smaller than $1$, even though the second moment of the EESD of $H_n$ converges to $1$ (see Remark~\ref{EESD_second_moments}).
This discrepancy indicates that the method of moments cannot be applied in this case.
A likely explanation for this phenomenon is the presence of a small number of outlier eigenvalues that contribute non-negligibly to the second moment while not affecting the ELSD.
Such outliers may arise from the covariance structure of $H_n$, in analogy with phenomena observed for finite-rank perturbations of random matrices, known as \emph{Baik--Ben Arous--Péché phase transitions}~\cite{BBP}.
In the uniform case, the presence of outliers is investigated in~\cite{ipergrafiuniformi}.
\end{Remark}

\subsection{Gaussianization}
Theorem~\ref{thm_semicircle}, concerning convergence to the semicircle law, relies on an intermediate result, which we refer to as the Gaussianization step.
We first formalize the notion of \emph{Gaussianization} in this context.
Recall that, given a probability measure $\mu$ on $\mathbb{R}$, its \emph{Stieltjes transform} is the analytic function $S_{\mu}\colon \mathbb{C}_+ \rightarrow \mathbb{C}_+$ defined by
\begin{equation*}
   S_\mu(z) \coloneqq \int_{\mathbb{R}} \frac{1}{x-z} \, \de \mu(x),
   \qquad z \in \mathbb{C}_+,
\end{equation*}
where $\mathbb{C}_+ \coloneqq \{ z \in \mathbb{C} \colon \operatorname{Im} z > 0 \}$. 

\begin{Definition}[$\bm{Z}$-Gaussianizable model]
Let $\bm{Z}=(Z_1,\dots,Z_M)$ be a vector of independent Gaussian components, and let $\bm{X}$ be the random vector defined in~\eqref{X}.
We say that $\bm{X}$ is \emph{$\bm{Z}$-Gaussianizable} if
\begin{equation*}    
    \lim_{n\to\infty} \biggl|
    S_{\bar{\mu}_{H_n(\bm{X})}}(z)-S_{\bar{\mu}_{H_n(\bm{Z})}}(z) 
    \biggr|=0
    \qquad \text{for any } z \in \mathbb{C}_+.
\end{equation*}
\end{Definition}

A well-known fact is that weak convergence of measures is equivalent to pointwise convergence of their Stieltjes transforms on $\mathbb{C}_+$.
Therefore, for Gaussianizable models, the expected limiting spectral distribution, provided it exists, is the same for both the original model and its Gaussianized counterpart. 

We now introduce the main working assumption under which, as will be stated in Theorem~\ref{thm_Gaussianization}, our model is Gaussianizable.
We refer to this as the \emph{Pastur-type condition}, as it is analogous to a condition for convergence to the standard semicircle law established by Pastur~\cite{Pastur} in the context of Wigner matrices.

\begin{Assumption}[Pastur-type condition] \label{Pasturconditions}
{Let $\bm{X}$ be the random vector associated with the hypergraph model, as defined in~\eqref{X}.
Let $\bm{Z}=(Z_1,\dots,Z_M)$ be a vector of independent Gaussian random variables, where each $Z_l$ has the same mean and variance of the corresponding $X_l$, i.e.,
\begin{equation} \label{Z}
        Z_l \sim \mathcal{N}(0,\sigma_i^2) \qquad \text{if } |e_l| = r_i.
\end{equation}}
\begin{enumerate}
\item We assume that $\bm{X}$ satisfies the following Pastur-type condition: for every $\epsilon > 0$,
\begin{equation}
\label{P}
    \sum_{l=1}^{M} \mathbb{E}\!\left[
            X_l^2 \,\mathds{1}\!\left(|X_l| > \epsilon K_n\right)
        \right]
        = o\!\left(
            \frac{n^2 \sigma^2}{r_{\max}^4}
        \right), \qquad \text{as } n \to \infty,
\end{equation}
where $r_{\max} \coloneqq \max\{r_i \colon 1 \leq i \leq k\}$, and
\begin{equation} \label{K_n}
    K_n \coloneqq \frac{\sqrt{n \sigma^2}}{r_{\max}^6 \,\xi},
\end{equation}
with $\xi$ defined in \eqref{xi}.
\item We assume that $\bm{Z}$ also satisfies the Pastur-type condition~\eqref{P}, in the sense that~\eqref{P} holds with each $X_l$ replaced by $Z_l$.
\end{enumerate}
\end{Assumption}

\begin{Remark}[Tail conditions for Gaussianization]
Since the vector $\bm{X}$ consists of $k$ groups of i.i.d.\ random variables, condition~\eqref{P} is equivalent to $k$ tail conditions, one for each group.
\end{Remark}

\begin{Remark}[Uniform hypergraphs and the Pastur-type condition]
In the $r_1$-uniform case ($k=1$), we have $\sigma^2 = \binom{{n-2}}{r_1-2}\sigma_1^2$, $r_{\mathrm{max}}=r_1$, and $w_1=1$, which yields $\xi=1/r_1^2$ and $K_n = \sqrt{n\sigma^2}/r_1^4$.
Therefore, condition~\eqref{P} coincides with the Pastur-type condition established in~\cite[{Assumption~2.1}]{ipergrafiuniformi} for uniform hypergraphs.
\end{Remark}

{We now state our Gaussianization result under the Pastur-type condition and, as a special case, under the previously introduced non-sparsity condition, which is more restrictive but easier to check.}

\begin{theorem}[Gaussianization] \label{thm_Gaussianization}
{Let $\bm{X}$ be the random vector associated with the hypergraph model, as defined in~\eqref{X}, and let $\bm{Z}$ be a random vector with independent Gaussian components, as defined in~\eqref{Z}.}
\begin{enumerate}
   \item Assume that both $\bm{X}$ and $\bm{Z}$ satisfy the Pastur-type condition~\eqref{P}.
   Then, $\bm{X}$ is $\bm{Z}$-Gaussianizable.
    \item Assume that $1-p_i=\Theta(1)$, for $1 \leq i \leq k$.
    Suppose moreover that the non-sparsity condition~\eqref{D} holds.
    Then both $\bm{X}$ and $\bm{Z}$ satisfy the Pastur-type condition~\eqref{P}; consequently, $\bm{X}$ is $\bm{Z}$-Gaussianizable. 
    \end{enumerate}
\end{theorem}

\section{Discussion} \label{sec_discussion}

\subsection{Dominant and balanced regimes} \label{sec_regimes}
In our results, two key quantities emerge: $\xi$ (defined in~\eqref{xi}), which appears in the conditions for Gaussianization, and $s^2$, the variance of the semicircle law (defined in~\eqref{variance_semicircle}).
Each of these is expressed as a convex combination, with common coefficients $w_i \in [0,1]$, of terms depending on a single hyperedge size.
The coefficients are functions of the full set of parameters $(\bm{r},\bm{p})$ of our model and can be interpreted as weights assigned to each type of hyperedge.
Interpreting the results therefore requires understanding the behavior of these weights.

To simplify the analysis, let $k=2$.
In this case, since the coefficients satisfy the constraint $w_1 + w_2 = 1$, it suffices to study $w_1$.
To this end, it is convenient to express $w_1$ in terms of the average degrees.
Note that
\begin{equation*}
    w_1 = 
    \lim_{n \to \infty} 
    \frac{1}{1+ \frac{\binom{n-2}{r_2-2}p_2(1-p_2)}{\binom{n-2}{r_1-2}p_1(1-p_1)}}.
\end{equation*}
Since
\begin{align*}
   \binom{n-2}{r_i-2} p_i & = \frac{r_i-1}{n-1} d_i,
\end{align*}
we can write
\begin{equation*}
    w_1 = \lim_{n \to \infty} \frac{1}{1+ rdq},
\end{equation*}
where
\begin{equation} 
    r \coloneqq \frac{r_2-1}{r_1-1}, \qquad d \coloneqq \frac{d_2}{d_1}, \qquad q \coloneqq \frac{1-p_2}{1-p_1}.
\end{equation}

\begin{table}
\centering
\renewcommand{\arraystretch}{3}
\begin{tabular}{|c|c|c|c|}
\hline
\makecell[c]{\textbf{$r_1$-degree} \\ \textbf{versus} \\ \textbf{$r_2$-degree}}
& \makecell[c]{\textbf{Sufficient} \\ \textbf{condition for} \\\textbf{Gaussianization}} 
& \makecell[c]{\textbf{Hyperedge} \\ \textbf{dominance}}
& \makecell[c]{\textbf{Variance of the limiting} \\ \textbf{semicircle law}} \\
\hline
$d_1 \gg d_2$
& $\begin{aligned} d_1 \gg {n}^{9} \end{aligned}$
& $r_1$-dominant: $w_1=1$
& $s^2=(1-c_1)^2$ \\
\hline 
$d_1 = \Theta(d_2)$
& \makecell[c]{$
d_1 \gg {n}^{9}$ \\ $\iff \, d_2 \gg {n}^{9}
$}
& Balanced: $0 < w_1< 1$ 
& $s^2= w_1(1-c_1)^2 + w_2(1-c_2)^2$ \\ 
\hline
$d_1 \ll d_2$
& $\begin{aligned} d_2 \gg {n}^{9} \end{aligned}$
& $r_2$-dominant: $w_1=0$
& $s^2=(1-c_2)^2$ \\
\hline
\end{tabular}
\caption{Illustration of Theorems~\ref{thm_semicircle} and~\ref{thm_Gaussianization} for $k=2$, where the hyperedge sizes $r_1$ and $r_2$ are of the same magnitude and scale linearly with $n$, with constants $0<c_1 < c_2<1$.
The table shows how the non-sparsity condition for convergence to the semicircle law (see Remark~\ref{same_magnitude_condition}) and the corresponding variance vary depending on the behavior of $d_1$ relative to $d_2$.}
\label{tab:regimes-comparison-same-magnitude}
\end{table}
 
Since we assume that the connection probabilities are bounded away from $1$ for large $n$ (see the hypotheses of Theorem~\ref{thm_semicircle}), the ratio $q$ does not play an asymptotically significant role.
Accordingly, depending on the magnitude of $d$ relative to $r$, we identify three \emph{dominance regimes}, which lead to different variances of the limiting semicircle law when $c_1 \neq c_2$ (to fix the ideas, suppose $c_1 < c_2$):
\begin{itemize}
\item \emph{$r_1$-dominant regime}. If $d \ll r^{-1}$, then $w_1=1$, and the contribution of the $r_2$-edges is negligible:
\begin{equation*}
   s^2 = (1 - c_1)^2.
\end{equation*} 
\item \emph{Balanced regime}. If $d = \Theta(r^{-1})$, then $0 < w_1 < 1$, and both types of hyperedges contribute to the limit:
\begin{equation*}
   (1 - c_2)^2 < s^2 < (1 - c_1)^2.
\end{equation*}
\item \emph{$r_2$-dominant regime}. If $d \gg r^{-1}$, then $w_1=0$, and the contribution of the $r_1$-edges is negligible:
\begin{equation*}
   s^2 = (1 - c_2)^2.
\end{equation*} 
\end{itemize}

Our main results, Theorems~\ref{thm_semicircle} and~\ref{thm_Gaussianization}, are fully general and apply to all asymptotic scalings of the parameters $r_1$ and $r_2$. 
We have two distinct possibilities: 
\begin{enumerate}
\item Hyperedges of the same magnitude, i.e., $r_1 = \Theta(r_2)$;
\item Hyperedges of different magnitude, e.g., $r_1 \ll r_2$. 
\end{enumerate}
Table~\ref{tab:regimes-comparison-same-magnitude} illustrates the aforementioned dominance regimes for hyperedges of the same magnitude (i.e., when $0<c_1<c_2<1$).

\subsection{Corollaries and showcase examples} \label{examples}

We now present some special cases of Theorem~\ref{thm_semicircle}.

Recalling that the average degrees are defined as $d_i = \binom{n-1}{r_i-1} p_i$ (see~\eqref{defaveragedegrees}), an asymptotic analysis of the binomial coefficients allows us to express the conditions for convergence to the semicircle law explicitly in terms of the asymptotic behaviour of the connection probabilities.
We consider the two `extreme' regimes in which the hyperedge sizes either scale linearly with $n$ or remain independent of $n$.

\begin{corollary}[Hyperedge sizes linear in $n$] \label{corollary_linear_sizes}
     Let $I(c) \coloneqq -c \ln c -(1-c) \ln(1-c)$ be the entropy of a Bernoulli random variable with parameter $c$.
     For each $1 \leq i \leq k$, assume that $r_i/n \to c_i \in (0,1)$ and $1-p_i=\Theta(1)$. Suppose moreover that there exists $i \in \{1,\dots,k\}$ such that 
    \begin{equation*}
        \frac{d_i}{n^9} \to \infty \quad\iff\quad p_i \gg \frac{n^{19/2}}{e^{I(c_i)n}}.
    \end{equation*}
    Then 
    \begin{equation*}
        \bar{\mu}_{H_n(\bm{X})} \xrightarrow{n\to\infty} \nu_{s^2} \qquad{\text{weakly}},
    \end{equation*}
    where $s^2$ is defined in~\eqref{variance_semicircle}.
\end{corollary}
\begin{proof}
    The sufficiency of the condition on the average degrees follows from Remark~\ref{same_magnitude_condition}.
    The corresponding requirement in terms of the connection probabilities follows from the asymptotics of the binomial coefficient, obtained via Stirling’s approximation for the factorial.
    Indeed, if $k \sim cn$, where $c \in (0,1)$, then, $\binom{n}{k} \sim \frac{\exp(I(c)n)}{\sqrt{2\pi nc(1-c)}}$.
\end{proof}

\begin{Example}
Let $k=2$.
Choose $r_1 = c_1n$ and $r_2=c_2 n$, with $c_1 \in (0,1)$, $c_1 \neq 1/2$, and $c_2 = 1/2$.
Assume that the hypotheses of Corollary~\ref{corollary_linear_sizes} are satisfied, so that convergence to the semicircle law holds.
Since the hyperedges have the same order of magnitude, the dominant and balanced regimes are determined by comparing $d_1$ and $d_2$ (see Table~\ref{tab:regimes-comparison-same-magnitude}).
By the asymptotics of the binomial coefficients, we have
\begin{equation*}
    \frac{d_2}{d_1} = \frac{\binom{n-1}{r_2-1}p_2}{\binom{n-1}{r_1-1}p_1} = \Theta\left(\frac{p_2}{p_1}e^{(I(1/2)-I(c_1))n}\right).
\end{equation*}
Since the entropy function $I(c)$ attains its unique maximum at $c=1/2$, we have $I(1/2) > I(c_1)$.

Consider the homogeneous case $p_1 = p_2$.
Then $d_2 \gg d_1$, hence, despite the homogeneity, the model is always in an $r_2$-dominant regime: $w_2=1$ and $s^2=1/4$.

To obtain a balanced regime, $p_2$ must be chosen significantly smaller than $p_1$, specifically with exponential decay: $p_2 = \Theta(e^{-(I(1/2)-I(c_1))n}p_1)$.
\end{Example}

\begin{corollary}[Hyperedge sizes fixed] \label{corollary_fixed_sizes}
    For each $1 \leq i \leq k$, assume that $r_i$ does not depend on $n$ and $1-p_i=\Theta(1)$.
    Suppose moreover that there exists $i \in \{1,\dots,k\}$ such that 
    \begin{equation*}
    d_i \to \infty \quad\iff\quad p_i \gg n^{1-r_i}.
    \end{equation*}
    Then
    \begin{equation*}
        \bar{\mu}_{H_n(\bm{X})} \xrightarrow{n\to\infty} \nu_{1} 
        \qquad{\text{weakly}}.
    \end{equation*}
\end{corollary}
\begin{proof}
Under the stated assumptions, we have $c_i = 0$ for $1 \leq i \leq k$, whence $s^2 = \sum_{i=1}^{k} w_i = 1$.
The sufficiency of the condition on the average degrees follows from Remark~\ref{same_magnitude_condition}, while the corresponding condition on the connection probabilities follows from the asymptotic estimate of the binomial coefficient.
Indeed, if $k$ is fixed, then $\binom{n}{k} = \Theta(n^k)$.
\end{proof}

\begin{Remark}\label{rem:SparseCase}
It is interesting to compare Corollary~\ref{corollary_fixed_sizes} with the $r$-uniform case with average degree $d_{\mathrm{avg}}$, for fixed $r$ (see Table~\ref{tab:comparison_uniform}).
The problem of identifying the limiting spectral distribution has been studied in different density/sparsity regimes:
\begin{itemize}
\item If $d_{\mathrm{avg}} \to \infty$, the LSD is the standard semicircle law $\nu_{1}$; see~\cite{ipergrafiuniformi}.
\item If $d_{\mathrm{avg}} \to \lambda > 0$, the LSD depends on $r$ and $\lambda$, and is denoted by $\Gamma_{r-1,\lambda}$; see~\cite{sparseuniform}.
\end{itemize}
The characterization of the distribution in the sparse case remains largely implicit, although some of its properties are described in~\cite[Prop.~5,6]{lsdsparse}: for $r \geq 3$, it is unbounded, asymmetric about $0$, and, under suitable scaling, converges in distribution to the standard semicircle law as $\lambda \to \infty$.

When $r = 2$, the model reduces to the classical Erd\H{o}s--R\'enyi sparse graph, and more detailed results are available for the limiting measure $\Gamma_{1,\lambda}$.
This line of research was initiated by the seminal work~\cite{bauer_and_golinelli}, which provided numerical predictions for several features of $\Gamma_{1,\lambda}$, later established analytically.
The existence of the limiting spectral distribution was established in~\cite{khorunzhy2004eigenvalue}.
The convergence in probability was later strengthened to an almost sure convergence in~\cite{Bordenave_Lelarge_Salez}.
The measure $\Gamma_{1,\lambda}$ turns out to have a dense set of atoms; see~\cite{Salez}.
Moreover, it has a continuous component if and only if $\lambda > 1$; see~\cite{Bordenave_Sen_Virag}.
An explicit formula for the atomic mass at $0$ was proved in~\cite{Bordenave_Lelarge_Salez}.
Finally, $\lambda = e$ is the threshold for the emergence of so-called \emph{extended states} (roughly speaking, an absolutely continuous part) at $0$ in $\Gamma_{1,\lambda}$; see~\cite{Coste_and_Salez}.
\end{Remark}

\begin{table}
\centering
\renewcommand{\arraystretch}{3}
\begin{tabular}{|c|c|c|c|}
\hline
\makecell[c]{\textbf{Features of} \\ \textbf{$r_1$-edges}} 
& \makecell[c]{\textbf{$r_1$-uniform} \\ \textbf{(E)LSD}} 
& \makecell[c]{\textbf{Features of} \\ \textbf{$r_2$-edges}}
& \makecell[c]{\textbf{Non-uniform} \\ \textbf{ELSD}} \\
\hline
\makecell[c]{$d_1 \to \infty$ \\ $p_1 \gg n^{1-r_1}$}
& $\nu_1$
& irrelevant
& $\nu_{1}$ \\
\hline
\multirow{2}{*}{\makecell[c]{$d_1 \to \lambda_1$ \\ $p_1=\Theta(n^{1-r_1})$}}
& \multirow{2}{*}{\makecell[c]{$\Gamma_{r_1-1,\lambda_1}$}}
& \makecell[c]{$d_2 \to \infty$ \\ $p_2 \gg n^{1-r_2}$}
& $\nu_1$ \\
\cline{3-4}
&
& \makecell[c]{$d_2 \to \lambda_2$ \\ $p_2=\Theta(n^{1-r_2})$}
& \textbf{?} \\
\hline
\end{tabular}
\caption{Comparison between the results obtained in the uniform and non-uniform settings for $k=2$ and fixed hyperedge sizes.
If one of the uniform subhypergraphs is sparse, then it does not contribute to the limit.
Theorem~\ref{thm_semicircle} does not cover the case in which both subhypergraphs are sparse.
Identifying the limit in this regime remains an open problem and is left for future work.}
\label{tab:comparison_uniform}
\end{table}

Corollary~\ref{corollary_fixed_sizes} applies, in particular, to the superposition of Erd\H{o}s--R\'enyi random graphs of Example~\ref{def_superposition}.
\begin{corollary}[Superposition of Erd\H{o}s--R\'enyi random graphs] \label{superposition_result}
    For each $1 \leq i \leq k$, assume that $r_i = 2$ and $1-p_i=\Theta(1)$.
    Suppose moreover that there exists $i \in \{1,\dots,k\}$ such that $np_i \to \infty$.
    Then
    \begin{equation*}
        \bar{\mu}_{H_n(\bm{X})} \xrightarrow{n\to\infty} \nu_{1} 
        \qquad{\text{weakly}}.
    \end{equation*}
\end{corollary}

As a final immediate corollary of Theorem~\ref{thm_semicircle}, we recover a known result due to~\cite{ipergrafiuniformi} for uniform hypergraphs.
For classical Erd\H{o}s--R\'enyi random graphs (i.e., $r_1=2$), this was established in~\cite[Corollary 1.2]{erclassico}. 

\begin{corollary}[Uniform hypergraphs] \label{corollary_uniform_hypergraphs}
    Let $k=1$.
    Suppose that $r_1/n \to c \in [0,1)$, $1-p_1=\Theta(1)$, and $d_1/r_1^9 \to \infty$.
    Then
    \begin{equation*}
        \bar{\mu}_{H_n(\bm{X})} \xrightarrow{n\to\infty} \nu_{(1-c)^2} 
        \qquad{\text{weakly}}. 
    \end{equation*}
\end{corollary}

\section{Proofs}
\label{proofs}
\subsection{Gaussianization: proof of Theorem~\ref{thm_Gaussianization}}

This section is devoted to the proof of the Gaussianization result, Theorem~\ref{thm_Gaussianization}. 

We begin by outlining the strategy of the proof. 
First, in Lemma~\ref{applicazioneChatterjee}, we bound the distance between the Stieltjes transforms of the EESDs of the random matrix $H_n(\bm{X})$ associated with the hypergraph model and of its Gaussianized counterpart $H_n(\bm{Z})$, where $\bm{Z}$ is defined in~\eqref{Z}.
To this end, we apply a general result due to Chatterjee~\cite{Chatterjee}, which extends Lindeberg’s method for the central limit theorem.
Next, in Proposition~\ref{PasturGaussianization}, we identify a condition on the random vectors $\bm{X}$ and $\bm{Z}$ under which the error term in this bound vanishes, namely the Pastur-type condition~\eqref{P}.
Finally, in Proposition~\ref{prop_D1_D2}, we use standard probabilistic inequalities to derive a condition that is stronger than the Pastur-type one but easier to verify in practice, namely the non-sparsity condition~\eqref{D}.

Our strategy is similar to that of~\cite{ipergrafiuniformi}, although there are some notable differences.
For instance, in the derivation of the Pastur-type condition in the proof of Proposition~\ref{PasturGaussianization}, the constant $K_n$ depends on the connection probabilities, whereas in the uniform case it does not.
Nevertheless, the definition reduces to that of the uniform case when there is only one hyperedge size, i.e., when $k=1$.
Moreover, the Gaussianization step requires assuming that $\bm{Z}$ satisfies the Pastur-type condition, a property that holds in the uniform setting independently of the model parameters.
For this reason, in Lemma~\ref{PasturZ}, we derive a stronger assumption ensuring that $\bm{Z}$ satisfies the Pastur-type condition.
Interestingly, this condition turns out to coincide with the one ensuring that $\bm{X}$ satisfies the Pastur-type condition (see Lemma~\ref{PasturX}).

\subsubsection{Derivation of the Pastur-type condition}
We first establish the sufficiency of the Pastur-type condition for Gaussianization.
The proof relies on the following result.

\begin{theorem}[Chatterjee~\cite{Chatterjee}] \label{theoremChatterjee}
Let $\bm{Y} = (Y_1,\dots,Y_M)$ and $\bm{W} = (W_1,\dots,W_M)$ be two independent random vectors defined on the same probability space.
Assume that the random variables $Y_i$ and the $W_i$ are independent, take values in some open interval $I$, have finite second moments, and satisfy $\mathbb{E}[Y_i] = \mathbb{E}[W_i]$ and $\mathbb{E}[Y_i^2] = \mathbb{E}[W_i^2]$ for each $i$.

Let $f \colon I^M \rightarrow \mathbb{R}$ be thrice differentiable in each argument, and set $U \coloneqq f(\bm{Y})$ and $V \coloneqq f(\bm{W})$.
Then, for any thrice differentiable $g\colon \mathbb{R} \rightarrow \mathbb{R}$ and any $K > 0$, we have
\begin{align*}
  |\mathbb{E}g(U) - \mathbb{E}g(V)| 
  & \leq C_1(g)\lambda_2(f)
  \sum_{i=1}^{M} \biggl[\mathbb{E}\!\left[Y_i^2 \mathds{1}(|Y_i| > K) \right] + \mathbb{E}[W_i^2 \mathds{1}(|W_i| > K)] \biggr] \\
   & + C_2(g)\lambda_3(f) \sum_{i=1}^{M} \biggl[\mathbb{E}\!\left[|Y_i|^3 \mathds{1}(|Y_i| \leq K) \right] + \mathbb{E}\!\left[|W_i|^3 \mathds{1}(|W_i| \leq K) \right] \biggr],
\end{align*}
where $C_1(g) \coloneqq ||g'||_{\infty}+||g''||_{\infty}$, $C_2(g)\coloneqq \frac{1}{6}||g'||_{\infty}+\frac{1}{2}||g''||_{\infty}+\frac{1}{6}||g'''||_{\infty}$, and 
\begin{equation} \label{def_lambda}
    \lambda_s(f) \coloneqq \sup\left\{\left|\diffp[q]{f(\bm{x})}{x_i}\right|^{\frac{s}{q}} \colon 1 \leq i \leq M, \;\; 1 \leq q \leq s, \;\; \bm{x} \in I^M\right\}.
\end{equation}
\end{theorem}

We now see why this result is useful for our purposes.
{Notice first that the Stieltjes transform of the EESD of $H_n(\bm{X})$ can be written as
\begin{equation}
\label{eq:Stieltjes}
S_{\bar{\mu}_{H_n(\bm{X})}}(z)
= \E\left[\frac{1}{n}\text{Tr}R(\bm{X})\right],
\qquad z \in \mathbb{C}_+,
\end{equation}
where 
\begin{equation*}R(\bm{x})\coloneqq\left(H_n(\bm{x})-zI \right)^{-1}, \qquad \bm{x}\in\R^M,
\end{equation*}
is the resolvent of the matrix $H_n(\bm{x})$.
Therefore, fix $z\in\mathbb{C}_+$ and} consider the function
\begin{equation*}
   f(\bm{x})\coloneqq\frac{1}{n} \text{Tr}R(\bm{x}), \qquad {\bm{x}\in\R^M.}
\end{equation*}
With this choice of $f$, we obtain the following lemma.
The proof is a straightforward extension to the non-uniform setting of the arguments given in~\cite[Section 2.4]{Chatterjee} and~\cite[Proposition 3.1]{ipergrafiuniformi} in the context of Wigner matrices and adjacency matrices of uniform hypergraphs, respectively.

\begin{lemma} \label{applicazioneChatterjee}
Let $\bm{X}$ and $\bm{Z}$ be the random vectors defined in~\eqref{X} and~\eqref{Z}.
For any $K > 0$, we have
\begin{equation}\label{Stieltjes}
  \begin{split}
  \biggl|S_{\bar{\mu}_{H_n(\bm{X})}}(z)-S_{\bar{\mu}_{H_n(\bm{Z})}}(z) \biggr| 
  & \leq 2\lambda_2(f) \sum_{l=1}^{M} \biggl[\mathbb{E}\!\left[X_l^2 \mathds{1}(|X_l| > K) \right] + \mathbb{E}\!\left[Z_l^2 \mathds{1}(|Z_l| > K) \right] \biggr] \\
  & + \frac{1}{3}\lambda_3(f) \sum_{l=1}^{M} \biggl[\mathbb{E}\!\left[|X_l|^3 \mathds{1}(|X_l| \leq K) \right] + \mathbb{E}\!\left[|Z_l|^3 \mathds{1}(|Z_l| \leq K) \right] \biggr],
\end{split}
\end{equation}
where 
\begin{align} 
    & \lambda_2(f) = O\!\left(\frac{r_{\max}^4}{n^2\sigma^2}\right) \label{lambda2}, \\
    & \lambda_3(f) = O\!\left(\frac{r_{\max}^6}{n^{5/2}\sigma^{3}}\right) \label{lambda3},
\end{align}
$\sigma^2$ is defined in~\eqref{sigmasquare}, and $r_{\max}\coloneqq\max\{r_i \colon 1 \leq i \leq k\}$.
\end{lemma} 

\begin{proof}
Writing $f(\bm{x}) = \Re f(\bm{x}) + \mathrm{i} \Im f(\bm{x})$, we apply Theorem~\ref{theoremChatterjee} separately to the real and imaginary parts of $f$, with $g(x)=x$. Note first that, for any $1 \leq i \leq M$ and $q \in \N$, we have
\begin{equation*}
    \diffp[q]{\Re f(\bm{x})}{x_i} = \Re\left(\diffp[q]{f(\bm{x})}{x_i}\right),
    \qquad
    \diffp[q]{\Im f(\bm{x})}{x_i} = \Im\left(\diffp[q]{f(\bm{x})}{x_i}\right),
\end{equation*}
and hence
\begin{equation} \label{eq_bound_derivatives}
    \max\left\{\left|\diffp[q]{\Re  f(\bm{x})}{x_i} \right|, 
    \left|\diffp[q]{\Im f(\bm{x})}{x_i} \right|\right\}
    \leq 
    \left|\diffp[q]{f(\bm{x})}{x_i} \right|.
\end{equation}
From~\eqref{eq_bound_derivatives} and~\eqref{def_lambda} we obtain 
\begin{equation} \label{eq_lambda_bound}
 \max\{\lambda_s(\Re f),\lambda_s(\Im f)\} \leq \lambda_s(f), \qquad s \in \N.
\end{equation}
It then follows from~\eqref{eq:Stieltjes} and Theorem~\ref{theoremChatterjee} that
\begin{flalign*}
    \biggl|S_{\bar{\mu}_{H_n(\bm{X})}}&(z) - S_{\bar{\mu}_{H_n(\bm{Z})}}(z) \biggr|
    = \biggl|\mathbb{E}f(\bm{X}) - \mathbb{E}f(\bm{Z}) \biggr| & \\
    & \leq \biggl|\mathbb{E}\Re f(\bm{X}) - \mathbb{E}\Re f(\bm{Z}) \biggr| + \biggl|\mathbb{E}\Im f(\bm{X}) - \mathbb{E}\Im f(\bm{Z}) \biggr| & \\
    & \leq C_1(g)(\lambda_2(\Re f) + \lambda_2(\Im f)) \sum_{l=1}^{M} \biggl[\mathbb{E}\!\left[X_l^2 \mathds{1}(|X_l| > K) \right] + \mathbb{E}\!\left[Z_l^2 \mathds{1}(|Z_l| > K) \right]\biggr]  & \\
    & + C_2(g)(\lambda_3(\Re f) + \lambda_3(\Im f)) \sum_{l=1}^{M} \biggl[\mathbb{E}\!\left[|X_l|^3 \mathds{1}(|X_l| \leq K) \right] + \mathbb{E}\!\left[|Z_l|^3 \mathds{1}(|Z_l| \leq K) \right] \biggr] & \\
    & \leq 2\lambda_2(f) \sum_{l=1}^{M} \biggl[\mathbb{E}\!\left[X_l^2 \mathds{1}(|X_l| > K) \right] + \mathbb{E}\!\left[Z_l^2 \mathds{1}(|Z_l| > K) \right]\biggr] & \\
    & + \frac{1}{3}\lambda_3(f) \sum_{l=1}^{M} \biggl[\mathbb{E}\!\left[|X_l|^3 \mathds{1}(|X_l| \leq K) \right] + \mathbb{E}\!\left[|Z_l|^3 \mathds{1}(|Z_l| \leq K) \right] \biggr],
\end{flalign*}
where for the latter inequality we used $C_1(g)=1$, $C_2(g)=1/6$, and~\eqref{eq_lambda_bound} for $s \in \{2,3\}$.
This yields~\eqref{Stieltjes}.

To derive the asymptotic bounds of the constants $\lambda_2(f)$ and $\lambda_3(f)$, recall that
\begin{align*}
    \lambda_2(f) 
     & = \sup_{1 \leq l \leq M} \left \{\nm*{\pdv{f(\bm{x})}{x_l}}_{\infty}^{2},\nm*{\diffp[2]{f(\bm{x})}{x_l}}_{\infty} \right \}, \\
    \lambda_3(f) 
    & = \sup_{1 \leq l \leq M} \left \{\nm*{\pdv{f(\bm{x})}{x_l}}^{3},\nm*{\diffp[2]{f(\bm{x})}{x_l}}^{\frac{3}{2}},\nm*{\diffp[3]{f(\bm{x})}{x_l}} \right \}.
\end{align*}
Note that
\begin{equation*}
    \diffp[q]{f(\bm{x})}{x_l} =
    \frac{1}{n}\text{Tr}\diffp[q]{R(\bm{x})}{x_l}, \qquad q \geq 1.
\end{equation*}
By differentiating the relation
\[
\left(H_n(\bm{x})-zI \right)R(\bm{x})=I,
\]
one obtains (see~\cite[Section 2]{Chatterjee}) the following identities:
\begin{align*}
& \pdv{f(\bm{x})}{x_l} = -\frac{1}{n}\text{Tr}\left(\pdv{H_n(\bm{x})}{x_l}R^2(\bm{x}) \right), \\
& \diffp[2]{f(\bm{x})}{x_l} = \frac{2}{n}\text{Tr}\left(\pdv{H_n(\bm{x})}{x_l}R(\bm{x})\pdv{H_n(\bm{x})}{x_l}R^2(\bm{x}) \right), \\
& \diffp[3]{f(\bm{x})}{x_l} = -\frac{6}{n} \text{Tr} \left(\pdv{H_n(\bm{x})}{x_l}R(\bm{x})\pdv{H_n(\bm{x})}{x_l}R(\bm{x})\pdv{H_n(\bm{x})}{x_l}R^2(\bm{x}) \right).
\end{align*}
Now we want to bound these quantities.

We start with an upper bound for $\nm*{\pdv{f(\bm{x})}{x_l}}_{\infty}$.
Observe that 
\begin{equation*}
    \mathrm{Tr}\left(\pdv{H_n(\bm{x})}{x_l}R^2(\bm{x})\right) = \sum_{1 \leq u,v \leq n}\left(\pdv{H_n(\bm{x})}{x_l}   \right)_{uv}\left(R^2(\bm{x})\right)_{vu}.
\end{equation*}
Letting $z=a+\mathrm{i} b\in \mathbb{C}_+$, we have that the eigenvalues $\mu_1,\dots,\mu_n$ of $R(\bm{x})$ are bounded above by $1/|b|$.
Indeed, if $\lambda_1,\dots,\lambda_n$ are the eigenvalues of $H_n(\bm{x})$, then
\begin{equation*}
    \mu_i = \frac{1}{\lambda_i - z}, \qquad 1 \leq i \leq n.
\end{equation*}
Since $H_n(\bm{x})$ is symmetric, its eigenvalues are real, and hence
\begin{equation} \label{eig_resolvent}
    |\mu_i| = \frac{1}{\sqrt{(\lambda_i-a)^2+b^2}} \leq \frac{1}{|b|}, \qquad 1 \leq i \leq n.
\end{equation}
It follows from the spectral decomposition of $R^2(\bm{x})$ that $|\left(R^2(\bm{x})\right)_{uv}| \leq 1/b^2$, and hence
\begin{equation*}
    \left|\text{Tr}\left(\pdv{H_n(\bm{x})}{x_l}R^2(\bm{x})\right) \right| \leq
    \frac{1}{b^2}\sum_{1 \leq u,v \leq n}\left|\left(\pdv{H_n(\bm{x})}{x_l} \right)_{uv} \right|.
\end{equation*}
Observing that
\begin{equation*}
    \pdv{H_n(\bm{x})}{x_l} = \frac{1}{\sqrt{n\sigma^2}}Q_l,
\end{equation*}
we have
\begin{equation*}
    \sum_{1 \leq u,v \leq n}\left|\left(\pdv{H_n(\bm{x})}{x_l} \right)_{uv} \right| 
    = \frac{1}{\sqrt{n\sigma^2}} \sum_{1 \leq u,v \leq n} |(Q_l)_{uv}|.
\end{equation*}
{Recall that, for $u\neq v$, $(Q_l)_{uv}$ equals $1$ if the hyperedge $e_l$ contains both vertices $u$ and $v$, and equals $0$ otherwise.
Therefore, the sum on the right-hand side of the previous equation equals $\abs{e_l}(\abs{e_l}-1)$.
In contrast to the uniform case, our model allows hyperedges of different sizes, so $\abs{e_l}$ depends on the index $l$.
Nevertheless, we still have the bound
\begin{equation*}
\sum_{1 \leq u,v \leq n} (Q_l)_{uv}
\leq
r_{\max}(r_{\max}-1).
\end{equation*}}
Putting everything together, we obtain
\begin{equation}\label{bound1}
    \sup_{1 \leq l \leq M}\nm*{\pdv{f(\bm{x})}{x_l}}_{\infty} \leq \frac{r_{\max}(r_{\max}-1)}{b^2n^{3/2}\sigma}.
\end{equation}

For bounding norms of higher order derivatives of $f(\bm{x})$ we need some linear algebra preliminaries.
For an $n\times n$ complex matrix $B$, the \emph{Frobenius norm} of $B$ is defined by 
\begin{equation*}
    \displaystyle \nm*{B}_{\emph{F}} \coloneqq \sqrt{\sum_{1 \leq u,v \leq n}|B_{uv}|^2} = \sqrt{\mathrm{Tr}(BB^{*})},
\end{equation*}
while the \emph{operator norm} (or \emph{spectral norm}) of $B$ is given by
\begin{equation*}
    \nm*{B}_{\text{op}} 
    \coloneqq
    \sup_{\nm*{\bm{x}}=1}\nm*{B\bm{x}} = \sqrt{\lambda_{\max}(BB^{*})},
\end{equation*}
where $\nm*{\bm{x}}= \sqrt{\sum_{1 \leq i \leq n}x^2_{i}}$ denotes the Euclidean norm of a vector $\bm{x} \in \mathbb{R}^n$.
We need to compute these two norms for the matrix $\pdv{H_n(\bm{x})}{x_l}$.
As before, each of them differs according to the index $l$.
We begin with the Frobenius norm: if $|e_l|=r_i$, then   
\begin{equation}
\label{eq:boundFrobenius}
    \displaystyle \nm*{\pdv{H_n(\bm{x})}{x_l}}_{\text{F}} = \frac{1}{\sqrt{n\sigma^2}}\nm*{Q_l}_{\text{F}} =  
            \sqrt{\frac{r_i(r_i-1)}{n\sigma^2}}.
\end{equation}
For the operator norm, note first that in the $r$-uniform case one has $\nm*{Q_l}_{\mathrm{op}} = r-1$.
Indeed, after relabeling the vertices if necessary, we may assume that
\begin{equation*}
    Q_l= 
    \begin{pmatrix}
        J_r - I_r & 0 \\
        0 & 0
    \end{pmatrix},
\end{equation*}
where  $J_r$ denotes the $r \times r$ all-ones matrix and $I_r$ is the identity matrix of dimension $r$.
Hence, $J_r-I_r$ is the $r \times r$ matrix with all ones in the off-diagonal entries and zeros on the diagonal, i.e., the adjacency matrix of the complete graph on $r$ vertices.
It is well known that this matrix has $r-1$ eigenvalues equal to $-1$ and one eigenvalue equal to $r-1$.
Therefore, $\nm*{Q_l}_{\text{op}} = \nm*{J_{r}-I_{r}}_{\text{op}}=r-1$.
It follows that, if $|e_l|=r_i$, then
\begin{equation}
\label{eq:boundOpNorm}
    \nm*{\pdv{H_n(\bm{x})}{x_l}}_{\text{op}} = \frac{1}{\sqrt{n\sigma^2}}\nm*{Q_l}_{\text{op}} =  
            \frac{r_i-1}{\sqrt{n\sigma^2}}.
\end{equation}

In the remainder of the proof, we will use the following standard inequalities involving the norms introduced above (see, e.g.,~Wilkinson~\cite[pp. 55-58]{Wilkinson}):
\begin{enumerate}
    \item $|\text{Tr}(BC)| \leq \nm*{B}_{\text{F}}\nm* {C}_{\text{F}}$.
    \item If $B$ is symmetric, then $\max\{\nm*{BC}_{\text{F}},\nm*{CB}_{\text{F}}\} \leq \nm*{B}_{\text{op}} \nm*{C}_{\text{F}}$.
\end{enumerate}

{Using these inequalities repeatedly, together with identity~\eqref{eq:boundFrobenius} and the bound $\nm*{R(\bm{x})}_{\text{op}} \leq 1/b$ (see~\eqref{eig_resolvent}), we obtain, for $|e_l| = r_i$,
\begin{align*} 
  \biggl  |\mathrm{Tr} \biggl(\pdv{H_n(\bm{x})}{x_l}R(\bm{x}) 
   \pdv{H_n(\bm{x})}{x_l}R^2(\bm{x}) \biggr) \biggr|
   & \leq \nm*{\pdv{H_n(\bm{x})}{x_l}}^2_{F} \nm*{R(\bm{x})}^3_{\text{op}}
   \leq \frac{1}{b^3} \frac{r_i(r_i-1)}{n\sigma^2}.
\end{align*}}
This yields
\begin{equation}\label{bound2}
    \sup_{1 \leq l \leq M}\nm*{\diffp[2]{f(\bm{x})}{x_l}}_{\infty} \leq \frac{2}{b^3}\frac{r_{\max}(r_{\max}-1)}{n^2 \sigma^2}.
\end{equation}

{Similarly, using the norm inequalities repeatedly, together with~\eqref{eq:boundFrobenius}, \eqref{eq:boundOpNorm}, and the bound for $\nm*{R(\bm{x})}_{\text{op}}$, we obtain, for $|e_l| = r_i$,
\begin{align*}
\left|\mathrm{Tr} \biggl(
\pdv{H_n(\bm{x})}{x_l}R(\bm{x})
\pdv{H_n(\bm{x})}{x_l}  R(\bm{x})
\pdv{H_n(\bm{x})}{x_l}R^2(\bm{x})
\biggr)\right|
&\le
\nm*{\pdv{H_n(\bm{x})}{x_l}}_{\text{op}}
\nm*{\pdv{H_n(\bm{x})}{x_l}}_{F}^2
\nm*{R(\bm{x})}_{\text{op}}^4 \\
&\le
\frac{1}{b^4}
\frac{r_i(r_i-1)^2}{{(n\sigma^2)}^{3/2}}.
\end{align*}}
It follows that
\begin{equation}\label{bound3}
    \sup_{1 \leq l \leq M}\nm*{\diffp[3]{f(\bm{x})}{x_l}}_{\infty} 
    \leq \frac{1}{b^4}\frac{6 r_{\max}(r_{\max}-1)^2}{n^{5/2} \sigma^{3}}.
\end{equation}
Using bounds~\eqref{bound1}, \eqref{bound2} and~\eqref{bound3}, we obtain
\begin{align*}
    \lambda_2(f) & \leq 2 \max \left(\frac{1}{b^3},\frac{1}{b^4} \right) \frac{r_{\max}^2(r_{\max}-1)^2}{n^2\sigma^2}, \\
    \lambda_3(f) & \leq 6 \max \left(\frac{1}{b^6},\frac{1}{b^{9/2}},\frac{1}{b^4} \right) \frac{r_{\max}^3(r_{\max}-1)^3}{n^{5/2}\sigma^{3}}.
\end{align*}
This yields~\eqref{lambda2} and~\eqref{lambda3}, and completes the proof.
\end{proof}

\begin{Remark}
Observe that the bounds for the constants $\lambda_2(f)$ and $\lambda_3(f)$ coincide with those in the {$r$-uniform case (with $r = r_{\max}$).}
The effect of allowing hyperedges of varying sizes appears only through the quantity $\sigma$.
\end{Remark}

\begin{proposition}[Gaussianization under the Pastur-type condition]\label{PasturGaussianization}
Suppose $\bm{X}$ and $\bm{Z}$ both satisfy the Pastur-type condition in Assumption~\ref{Pasturconditions}.
Then, for any $z \in \mathbb{C}_+$, we have
    \begin{equation*}    
        \lim_{n\to\infty} \biggl|S_{\bar{\mu}_{H_n(\bm{X})}}(z)-S_{\bar{\mu}_{H_n(\bm{Z})}}(z) \biggr|=0.
    \end{equation*}
\end{proposition}

\begin{proof}
From Lemma~\ref{applicazioneChatterjee}, it follows that, for any $K > 0$,
\begin{equation*}
\begin{split}
   \biggl|S_{\bar{\mu}_{H_n(\bm{X})}}(z)-S_{\bar{\mu}_{H_n(\bm{Z})}}(z) \biggr| 
   & \leq O\!\left(\frac{r_{\max}^4}{n^2\sigma^2}\right) \sum_{l=1}^{M} \biggl[\mathbb{E}\!\left[X_l^2 \mathds{1}(|X_l| > K) \right] + \mathbb{E}\!\left[Z_l^2 \mathds{1}(|Z_l| > K) \right] \biggr] \\
   & + O\!\left(\frac{r_{\max}^6}{n^{5/2}\sigma^3}\right) \sum_{l=1}^{M} \biggl[\mathbb{E}\!\left[|X_l|^3 \mathds{1}(|X_l| \leq K) \right] + \mathbb{E}\!\left[|Z_l|^3 \mathds{1}(|Z_l| \leq K) \right] \biggr].
\end{split}
\end{equation*}
We aim to prove that, as $n \to \infty$, the right-hand side vanishes for a suitable choice of $K$.
{Let $\epsilon > 0$ and $K=\epsilon K_n$, with $K_n$ as in~\eqref{K_n}.
It then follows immediately from Assumption~\ref{Pasturconditions} that
\begin{equation*}
    \frac{r_{\max}^4}{n^2 \sigma^2}  
    \sum_{l=1}^{M}\biggl[\mathbb{E}\!\left[X_l^2 \mathds{1}(|X_l| > \epsilon K_n) \right] + \mathbb{E}\!\left[Z_l^2 \mathds{1}(|Z_l| > \epsilon K_n) \right] \biggr] \longrightarrow 0. \\
\end{equation*}
The remainder of the proof is devoted to show that
\begin{equation}\label{eq:boundThirdMom}
        \frac{r_{\max}^6}{n^{5/2} \sigma^3} \sum_{l=1}^{M}\biggl[\mathbb{E}\!\left[X_l^3 \mathds{1}(|X_l| \leq \epsilon K_n) \right] + \mathbb{E}\!\left[Z_l^3 \mathds{1}(|Z_l| \leq \epsilon K_n) \right] \biggr] < {\tilde{C}} \epsilon
    \end{equation}
for {some $\tilde{C} > 0$} and $n$ sufficiently large.
Since $\epsilon$ is arbitrary, this concludes the proof.}

The key observation is that truncated third moments can be controlled by second moments.
Indeed, we use the following elementary inequality: for any random variable $Y$ and for any $K >0$,
\begin{equation*}
   \mathbb{E}\!\left[|Y|^3\mathds{1}(|Y| \leq K) \right] \leq K\mathbb{E}\!\left[Y^2 \right].
\end{equation*}
Applying this inequality to each $X_l$ and $Z_l$, and summing over $l=1,\dots,M$, we obtain:
\begin{equation} \label{boundK}
\frac{r_{\max}^6}{n^{5/2}\sigma^3}
\sum_{l=1}^{M}
\Big[
\mathbb{E}\big(|X_l|^3 \mathds{1}(|X_l|\le K)\big)
+\mathbb{E}\big(|Z_l|^3 \mathds{1}(|Z_l|\le K)\big)
\Big]
\le
2K\,\frac{r_{\max}^6}{n^{5/2}\sigma^3}
\sum_{i=1}^{k}
\binom{n}{r_i}\sigma_i^2.
\end{equation}
Thus, to establish the claim, it suffices to choose $K$ so that the right-hand side vanishes.
The problem is therefore reduced to controlling the asymptotic behavior of
\begin{equation} 
    \alpha_n \coloneqq \frac{r_{\max}^6}{n^{5/2} \sigma^3}  \sum_{i=1}^{k}
\binom{n}{r_i}\sigma_i^2.
\end{equation}
To this end, we introduce the shorthand
\begin{equation} \label{shorthand}
    B_i = \binom{n-2}{r_i-2}\sigma_i^2, \qquad 1 \leq i \leq k.
\end{equation}
Since
\begin{equation*}
    \binom{n}{r_i} = \frac{n(n-1)}{r_i(r_i-1)} \binom{n-2}{r_i-2} {=  \Theta\!\left(\frac{n^2}{r_i^2} \binom{n-2}{r_i-2}\right)}, 
\end{equation*}
we deduce\footnote{{Recall that, if $x_n,y_n,u_n,v_n$ are \emph{positive} sequences such that $x_n=\Theta(y_n)$ and $u_n=\Theta(v_n)$, then $x_n+u_n = \Theta(y_n+v_n)$.}} that
\begin{equation} \label{ciao1}
    \sum_{i=1}^{k}\binom{n}{r_i}\sigma_i^2 = {\Theta\!\left(n^2 \sum_{i=1}^{k}\frac{1}{r_i^2}B_i\right)} 
    \qquad\text{as } n\to\infty.
\end{equation}
Also note that, by the definition of $\sigma^2$ (see~\eqref{sigmasquare}), we have
\begin{equation} \label{ciao2}
    \sigma^2 = \sum_{i=1}^{k} B_i.
\end{equation}
Rewrite now $\alpha_n$ as
\begin{equation*}
    \alpha_n=\frac{r_{\max}^6}{\sqrt{n \sigma^2}} \rho,
\end{equation*}
where  
\begin{equation*}
    \rho = 
    \frac{\sum_{i=1}^{k}\binom{n}{r_i}\sigma_i^2}{n^2\sigma^2}.
\end{equation*}
Using~\eqref{ciao1} and~\eqref{ciao2}, we obtain
\begin{align*}
    \rho & = {\Theta\!\left(\sum_{i=1}^{k} \frac{1}{r_i^2} \left(\frac{B_i}{\sum_{j=1}^{k} B_j} \right)\right)}.
\end{align*}
{By Assumption~\ref{regularity_condition}}, for any $i=1,\dots,k$, the limit
\begin{equation*}
    w_i \coloneqq \lim_{n \to \infty} \frac{B_i}{\sum_{j=1}^{k} B_j}
\end{equation*}
exists.
Therefore, {$\rho = \Theta(\xi)$, where $\xi$ is} the following convex combination of $w_i$, {defined in~\eqref{xi}}:
\begin{equation*}
    \xi \coloneqq \sum_{i=1}^{k} w_i \frac{1}{r_i^2}.
\end{equation*}
In particular, we have
\begin{equation*}
    \alpha_n = {\Theta\!\left(\frac{r_{\max}^6 \xi}{\sqrt{n\sigma^2}}\right)}.
\end{equation*}
This suggests choosing $K=\epsilon K_n$, where $\epsilon>0$ and
\begin{equation*}
    K_n= \frac{\sqrt{n\sigma^2}}{r_{\max}^6\xi}.
\end{equation*}
With this choice of $K$, and using~\eqref{boundK}, we obtain
\begin{equation} \label{finalboundK}
    \frac{r_{\max}^6}{n^{5/2}\sigma^3}   \sum_{l=1}^{M}\biggl[\mathbb{E}\!\left[|X_l|^3 \mathds{1}(|X_l| \leq \epsilon K_n) \right] + \mathbb{E}\!\left[|Z_l|^3 \mathds{1}(|Z_l| \leq \epsilon K_n) \right] \biggr] \leq 2 \epsilon K_n \alpha_n =  2 \epsilon \frac{\rho}{\xi}. 
\end{equation}
{Since $\rho=\Theta(\xi)$, there exists a constant $C >0$ and $n_0 \in \mathbb{N}$ such that, for all $n \geq n_0$, $\rho/\xi \leq C$. 
Setting $\tilde{C}\coloneqq 2C$ yields~\eqref{eq:boundThirdMom}, thereby completing the proof.}
\end{proof}

\subsubsection{Non-sparsity condition}
The connection between the non-sparsity condition and the Pastur-type condition is formalized in the following result.

\begin{proposition}[Pastur-type and non-sparsity conditions] \label{prop_D1_D2}
Assume that $1-p_i = \Theta(1)$ for all $1 \leq i \leq k$, and that the non-sparsity condition~\eqref{D} holds. 
Then both $\bm{X}$ and $\bm{Z}$ satisfy the Pastur-type condition~\eqref{P}.
\end{proposition}

We prove Proposition~\ref{prop_D1_D2} for $\bm{X}$ and $\bm{Z}$ separately in the following two lemmas.

\begin{lemma} \label{PasturX}
Assume that $1-p_i = \Theta(1)$ for all $1 \leq i \leq k$, and that the non-sparsity condition~\eqref{D} holds.
Then $\bm{X}$ satisfies the Pastur-type condition~\eqref{P}.
\end{lemma}

\begin{proof}
Since the vector $\bm{X}$ consists of $k$ groups of i.i.d.\ random variables (see~\eqref{X}), we have
\begin{equation*}
\sum_{l=1}^{M} \mathbb{E}\!\left[
            X_l^2 \,\mathds{1}\!\left(|X_l| > \epsilon K_n\right)
        \right] = \sum_{i=1}^{k} \binom{n}{r_i} \mathbb{E}\!\left[
            \Tilde{X}_i^2 \,\mathds{1}\!\left(|\Tilde{X}_i| > \epsilon K_n\right)
        \right],
\end{equation*}
where $\Tilde{X}_i + p_i \sim \mathrm{Ber}(p_i)$ for $1 \leq i \leq k$.
Therefore, we need to show that, for all $\epsilon>0$,
\begin{equation}
\frac{r_{\max}^4}{n^2 \sigma^2} \,  \sum_{i=1}^{k} \binom{n}{r_i} \mathbb{E}\!\left[\Tilde{X}_i^2 \mathds{1}\!\left(|\Tilde{X}_i|>\epsilon K_n\right) \right] \xrightarrow{n\to\infty} 0.
\end{equation}
Fix $i \in \{1,\dots,k\}$ and note that
\begin{equation*}
    |\Tilde{X}_i| \leq \max\{p_i,1-p_i\} \leq 1.
\end{equation*}
Using Chebyshev's inequality, the fact that $\mathbb{E}\big[\Tilde{X}_i^2 \big] = \mathrm{Var}\big[\Tilde{X}_i\big] = p_i(1-p_i) \leq p_i$, and the definition~\eqref{K_n} of $K_n$, we thus obtain
\begin{equation*}
    \mathbb{E}\!\left[\Tilde{X}_i^2 \mathds{1}\!\left(|\Tilde{X}_i|>\epsilon K_n\right)\right]
    \le \mathbb{P}\left(|\Tilde{X}_i|>\epsilon K_n\right)
    \leq 
    \frac{\mathbb{E}\big[\Tilde{X}_i^2 \big]}{\epsilon^2 K_n^2}
    \leq \frac{r_{\max}^{12} \, \xi^2 p_i}{\epsilon^2 n \sigma^2}.
\end{equation*}

By assumption we have $1-p_i = \Theta(1)$ for all $1 \leq i \leq k$, hence there exists $C\in (0,\infty)$ such that $1-p_i\geq C$ for all $1 \leq i \leq k$.
Moreover, by the definition of $\sigma^2$ in~\eqref{sigmasquare} and of the average degree $d_i$ in~\eqref{defaveragedegrees}, we have
\begin{align*}
    \sigma^2 \coloneqq \sum_{i=1}^{k} \binom{n-2}{r_i-2}\sigma_i^2 
     = \sum_{i=1}^{k} \frac{r_i-1}{n-1}d_i(1-p_i).
\end{align*}
Therefore,
\begin{equation}
    \label{eq:boundC}
n \sigma^2
=n \sum_{i=1}^{k} \frac{r_i-1}{n-1}d_i(1-p_i)
\geq C \sum_{i=1}^k (r_i-1)d_i.
\end{equation}
Observing now that 
\begin{equation*}
   \binom{n}{r_i} p_i = \frac{n}{r_i} d_i
\end{equation*}
and putting everything together, we obtain
\begin{equation*}
\frac{r_{\max}^4}{n^2 \sigma^2}
\sum_{i=1}^{k} \binom{n}{r_i}
\mathbb{E}\!\left[
    \tilde{X}_i^2 \,\mathds{1}\!\left(|\tilde{X}_i|>\epsilon K_n\right)
\right]
\leq \frac{r_{\max}^{16} \, \xi^2}{\epsilon^2 n^2 \sigma^4} \sum_{i=1}^k \frac{d_i}{r_i}
\leq 
\frac{r_{\max}^{16}\,\xi^2 \sum_{i=1}^{k} \frac{d_i}{r_i}}
{\epsilon^2 C^2 \big(\sum_{i=1}^{k} (r_i-1) d_i\big)^2}.
\end{equation*}
The denominator in the latter ratio is $\Theta\big(\sum_{i=1}^{k} r_i d_i\big)^2$, and hence
\begin{equation*}
\frac{r_{\max}^4}{n^2 \sigma^2}
\sum_{i=1}^{k} \binom{n}{r_i}
\mathbb{E}\!\left[
    \tilde{X}_i^2 \,\mathds{1}\!\left(|\tilde{X}_i|>\epsilon K_n\right)
\right]
= O\!\left(\frac{r_{\max}^{16}\,\xi^2 \sum_{i=1}^{k} \frac{d_i}{r_i}}
{\big(\sum_{i=1}^{k} r_i d_i\big)^2}\right).
\end{equation*}
Now, if the non-sparsity condition~\eqref{D} holds, then the right-hand side above vanishes.
Hence, $\bm{X}$ satisfies the Pastur-type condition~\eqref{P}.
\end{proof}

\begin{lemma} \label{PasturZ}
Assume that $1-p_i = \Theta(1)$ for all $1 \leq i \leq k$, and that the non-sparsity condition~\eqref{D} holds. 
Then $\bm{Z}$ satisfies the Pastur-type condition~\eqref{P}.
\end{lemma}

\begin{proof}
As in the proof of Lemma~\ref{PasturX}, note that
\begin{equation*}
        \sum_{l=1}^{M} \mathbb{E}\!\left[
        Z_l^2 \,\mathds{1}\!\left(|Z_l| > \epsilon K_n\right)
        \right] 
        = \sum_{i=1}^{k} \binom{n}{r_i} \mathbb{E}\!\left[
        \Tilde{Z}_i^2 \,\mathds{1}\!\left(|\Tilde{Z}_i| > \epsilon K_n\right)
        \right],
\end{equation*}
where $\Tilde{Z}_i \sim \mathcal{N}(0,\sigma_i^2)$ for $1 \leq i \leq k$.
Therefore, we need to show that, for all $\epsilon>0$,
\begin{equation} \label{gaussiane_pastur}
\frac{r_{\max}^4}{n^2 \sigma^2} \,  \sum_{i=1}^{k} \binom{n}{r_i} \mathbb{E}\!\left[\Tilde{Z}_i^2 \,\mathds{1}\!\left(|\Tilde{Z}_i|>\epsilon K_n\right) \right] \xrightarrow{n\to\infty} 0.
\end{equation}
Fix $i \in \{1,\dots,k\}$ and note that
\begin{align*}
   \mathbb{E}\!\left[ \Tilde{Z}_i^2 \mathds{1}\!\left(|\Tilde{Z}_i|>\epsilon K_n\right) \right]
   &\le \left(\mathbb{E}|\Tilde{Z}_i|^4\right)^{1/2}
   \mathbb{P}\left(|\Tilde{Z}_i|>\epsilon K_n\right)^{1/2} \notag\\
   &\le \left(\mathbb{E}|\Tilde{Z}_i|^4\right)^{1/2} 
   \left(\mathbb{E}|\Tilde{Z}_i|^{4}\right)^{1/2} 
   (\epsilon K_n)^{-2} \notag\\
   &  = \frac{\mathbb{E}|\Tilde{Z}_i|^{4}}{\epsilon^{2}}
    \frac{(r_{\max}^6 \xi)^{2}}{n\sigma^2},
\end{align*}
where we used the Cauchy--Schwarz inequality for the first inequality, Markov's inequality for the second one, and the definition of $K_n$ in~\eqref{K_n} for the final equality.
Note that $\mathbb{E}|\Tilde{Z}_i|^4 = \sigma_i^{4}m_4$, where $m_4$ is the fourth moment of a standard normal random variable; {in particular, $m_4$ is finite and does not depend on $n$.}
Inserting the above estimate, we obtain
\[
\frac{r_{\max}^4}{n^2 \sigma^2} \,  \sum_{i=1}^{k} \binom{n}{r_i} \mathbb{E}\!\left[\Tilde{Z}_i^2 \mathds{1}\!\left(|\Tilde{Z}_i|>\epsilon K_n\right) \right]
\leq \frac{r_{\max}^{16} \, \xi^{2} m_4 \sum_{i=1}^{k} \binom{n}{r_i} \sigma_i^{4}}{\epsilon^2 n (n\sigma^2)^{2}}.
\]
Observe that    
\begin{align*}
    \binom{n}{r_i} \sigma_i^4
    = \frac{n}{r_i} d_i (1-p_i) \sigma_i^2
    \leq \frac{n}{r_i} d_i,
\end{align*}
since $1-p_i\leq 1$ and $\sigma^2_i=p_i(1-p_1)\leq 1$.
Combining this estimate with~\eqref{eq:boundC}, we obtain
\[
\frac{r_{\max}^4}{n^2 \sigma^2} \sum_{i=1}^{k} \binom{n}{r_i} \mathbb{E}\!\left[\Tilde{Z}_i^2 \mathds{1}\!\left(|\Tilde{Z}_i|>\epsilon K_n\right) \right]
\leq 
\frac{r_{\max}^{16}\,\xi^2 m_4 \sum_{i=1}^{k} \frac{d_i}{r_i}}
{\epsilon^2 C^2 \big(\sum_{i=1}^{k} (r_i-1) d_i\big)^2}.
\]
We now conclude the proof by repeating the argument used at the end of the proof of Lemma~\ref{PasturX}.
\end{proof}

\subsection{Convergence to the semicircle law: proof of Theorem~\ref{thm_semicircle}} \label{secsemicircle}
This subsection is dedicated to the proof of Theorem~\ref{thm_semicircle}.
{By Propositions~\ref{PasturGaussianization} and~\ref{prop_D1_D2}, it suffices to study the convergence to the semicircle law in the Gaussianized model.}
This is established in the result below.

The proof of the convergence to a semicircle law follows exactly the same argument of~\cite{ipergrafiuniformi} for uniform hypergraphs.
We outline their argument by briefly describing the structure of the random matrix $H_n(\bm{Z})$.
As for $H_n(\bm{X})$, each off-diagonal entry has mean zero and variance $1/n$; moreover, being a sum of independent Gaussian random variables, it is itself Gaussian, with distribution $\mathcal{N}(0,1/n)$.
How does then $H_n(\bm{Z})$ differ from a classical GOE matrix?
There are two main differences: firstly, the diagonal entries of $H_n$ are zero, whereas in a GOE matrix they are distributed as $\mathcal{N}(0,2/n)$; secondly, the entries of $H_n(\bm{Z})$ are correlated, whereas in a GOE matrix they are independent (up to symmetry).
Nevertheless, following the approach of~\cite{ipergrafiuniformi}, one can represent $H_n(\bm{Z})$ as a finite-rank perturbation of a suitably scaled GOE matrix with its diagonal removed.
From this, convergence to the semicircle law follows by using standard inequalities involving distances between real measures.

However, we still need an \emph{ad hoc} computation for the variance of the limiting semicircle law in the non-uniform case; see the proof below.

\begin{proposition}[ELSD, Gaussian case] \label{prop_Gaussian_case}
Suppose $r_i/n \rightarrow c_i \in [0,1)$, for $1 \leq i \leq k$.
Then, as $n \to \infty$,
\begin{equation*}
    \bar{\mu}_{H_n(\bm{Z})} \xlongrightarrow{n \to \infty} \nu_{s^2} 
    \qquad{\text{weakly}},
\end{equation*}
where the variance $s^2$ is defined in~\eqref{variance_semicircle}.
\end{proposition}

\begin{proof}
The proof of the convergence to a semicircle law is the same as in~\cite{ipergrafiuniformi}.
It remains to calculate the explicit expression of the variance $s^2$ of the limiting semicircle law in the non-uniform case.
This depends on the covariance structure of $H_n$ as follows:
\begin{equation} \label{theta_n}
   {s^2 = \lim_{n \rightarrow \infty} \theta_n^2, \qquad \theta_n^2 
   \coloneqq 1-2\gamma_n+\rho_n,}
\end{equation} 
where
\begin{equation*}
    \mathrm{Cov}((H_n)_{uv},(H_n)_{u'v'}) = 
        \begin{cases}
            \rho_n & \text{if } |\{u,v,u',v'\}|=4, \\[5pt]
            \gamma_n & \text{if } |\{u,v,u',v'\}|=3.
    \end{cases}
\end{equation*}
By~\eqref{covariance} and~\eqref{def_H}, we have
\begin{equation*}
    \rho_n = \frac{\sum_{i=1}^{k} \binom{n-4}{r_i-4} \sigma_i^2}{\sum_{i=1}^{k} \binom{n-2}{r_i-2} \sigma_i^2}, 
    \qquad 
    \gamma_n = \frac{\sum_{i=1}^{k} \binom{n-3}{r_i-3} \sigma_i^2}{\sum_{i=1}^{k} \binom{n-2}{r_i-2} \sigma_i^2}. 
\end{equation*}
To compute the limits of $\rho_n$ and $\gamma_n$, we introduce some shorthand notation.
For $i=1,\dots,k$, let 
\begin{equation*}
    {\zeta_i} \coloneqq \frac{(r_i-2)(r_i-3)}{(n-2)(n-3)}, \qquad
    {\eta_i} \coloneqq \frac{r_i-2}{n-2}, \qquad
    B_i \coloneqq \binom{n-2}{r_i-2} \sigma_i^2,
\end{equation*}
where $B_i$ was already defined in~\eqref{shorthand}.
Then 
\begin{align*}
    \rho_n = \frac{\sum_{i=1}^{k} \zeta_i B_i}{\sum_{i=1}^{k} B_i} = \sum_{i=1}^{k} \left(\frac{B_i}{\sum_{j=1}^{k} B_j} \right) \zeta_i.
\end{align*}
Since $\zeta_i \rightarrow c_i^2$ as $n \to \infty$ for all $i$, it follows that
\begin{equation} \label{lim_rho}
   \lim_{n \to \infty} \rho_n = \sum_{i=1}^{k} w_i c_i^2, 
\end{equation}
where the coefficients
\begin{equation*}
   w_i = \lim_{n \rightarrow \infty} \frac{B_i}{\sum_{j=1}^{k} B_j}, \qquad 1 \leq i \leq k, 
\end{equation*}
are defined as in~\eqref{weights}.
In a similar manner, 
\begin{equation*}
    \gamma_n = \frac{\sum_{i=1}^{k} \eta_iB_i}{\sum_{i=1}^{k} B_i}.
\end{equation*}
Since $\eta_i \rightarrow c_i$ as $n \to \infty$ for all $i$, it follows that
\begin{equation} \label{lim_gamma}
    \lim_{n \to \infty} \gamma_n = \sum_{i=1}^{k} w_i c_i.
\end{equation}
Consequently, combining~\eqref{theta_n}, \eqref{lim_rho}, and \eqref{lim_gamma}, and using that $\sum_{i=1}^{k} w_i=1$, we obtain
\begin{equation*}
s^2
= \lim_{n \to \infty} \theta^2_n  
= \lim_{n \to \infty} (1 - 2  \gamma_n +  \rho_n) 
= \sum_{i=1}^{k} w_i(1-c_i)^2,
\end{equation*}
as desired.
\end{proof}

\printbibliography

\end{document}